\documentclass{amsart}

\usepackage{amscd,amssymb,dsfont} 
\usepackage{url}
\usepackage{rotating}
\usepackage[utf8]{inputenc}
\usepackage{amsmath}

\usepackage[colorlinks=true]{hyperref}
\usepackage{mathrsfs}
\usepackage{bbm}
\usepackage{framed}

\usepackage{xcolor}

\newcommand{\red}[1]{\textcolor{black}{#1}} 

\newcommand{\dN}{d_{\tilde N}}

\usepackage{stmaryrd}
\usepackage{url}
\usepackage{pinlabel}

\usepackage{subfigure}

\usepackage[export]{adjustbox}
\usepackage{tikz-cd}

\makeatletter
\def\UTFviii@defined#1{%
	\ifx#1\relax
	!!FIXME!!%
	\else
	\expandafte‌​r#1%
	\fi
}
\makeatother

\DeclareFontFamily{OT1}{pzc}{}
\DeclareFontShape{OT1}{pzc}{m}{it}{<-> s * [1.10] pzcmi7t}{}
\DeclareMathAlphabet{\mathpzc}{OT1}{pzc}{m}{it}

\usepackage{syntonly}

\newcommand{\deltxt}[1]{}

\newtheorem{theorem}{Theorem}[section]
\newtheorem{corollary}[theorem]{Corollary}
\newtheorem{lemma}[theorem]{Lemma}
\newtheorem{proposition}[theorem]{Proposition}

\theoremstyle{definition}
\newtheorem{definition}{Definition}[section]
\newtheorem{remark}{Remark}[section]

\theoremstyle{remark}

\newcommand{\eval}[2][\right]{\relax
	\ifx#1\right\relax \left.\fi#2#1\rvert}

\newcommand\bR{{\mathbb{R}}}

\newcommand\bZ{{\mathbb Z}}

\newcommand\bV{{\mathbb V}}

\newcommand\Hom{{\rm Hom}}
\newcommand\dev{{\bf dev}}

\newcommand\clo{{\rm Cl}}

\newcommand\ra{\rightarrow}

\newcommand\emp{\emptyset}
\newcommand\eps{\epsilon}

\newcommand\Aff{{\mathbf{Aff}}}

\newcommand\Aut{{\mathbf{Aut}}}
\newcommand\Idd{{\rm I}}

\newcommand\bv{{\mathbf{v}}}

\newcommand\SL{{\mathsf{SL}}}

\newcommand\SO{{\mathsf{SO}}}

\newcommand\GL{{\mathsf{GL}}}

\newcommand{\vv}{\mathbf{v}}
\newcommand{\vw}{\mathbf{w}}

\newcommand{\Uu}{\mathbf{U}} 
\newcommand{\Uc}{\mathbf{UC}}

\newcommand{\llrrV}[1]{
	\left|\mkern-2mu\left|#1\right|\mkern-2mu\right|}

\begin{document}
	
	\title[Anosov-affine structures]{
		Complete affine manifolds 
		with Anosov holonomy groups} 

	\author[S. Choi]{Suhyoung Choi}
	\thanks{This research was supported by the Basic Science Research Program through the National Research Foundation
		of Korea(NRF) funded by the Ministry of Education(2019R1A2C108454412), 
		Author orcid number: 0000-0002-9201-8210}
	\address{Department of Mathematical Sciences \\ KAIST \\Daejeon 34141, South Korea}
	
	\email{schoi@math.kaist.ac.kr}
	
	%
	
	
	\date{\today}
	
	%
	%
	%
	%
	%
	%
	%
	%
	%
	
	
	
	
	
	%
	
	\begin{abstract}

Let $N$ be a complete affine manifold $\mathds{A}^n/\Gamma$ of dimension $n$, where $\Gamma$ is an affine transformation group acting on the complete affine space $\mathds{A}^n$, and $K(\Gamma, 1)$ is realized as a finite CW-complex. 
$N$ has a \emph{$k$-partially hyperbolic holonomy group} if the tangent bundle pulled back over the unit tangent bundle of a sufficiently large compact subset splits into expanding, neutral, and contracting subbundles along the geodesic flow, where the expanding and contracting
subbundles are $k$-dimensional with $k < n/2$.
In part 1, we will demonstrate that the complete affine $n$-manifold has a $P$-Anosov linear holonomy group for a parabolic subgroup $P$ of $\GL(n, \bR)$ if and only if it has a partially hyperbolic linear holonomy group. This had never been done over the full general linear group before this paper. Part 1 will primarily employ representation theory techniques.
In part 2, we demonstrate that if the holonomy group is partially hyperbolic of index $k$, where $k < n/2$, then $\mathrm{cd}(\Gamma) \leq n-k$. Moreover, if a finitely-presented affine group $\Gamma$ acts properly discontinuously and freely on $\mathds{A}^n$ with a $k$-Anosov linear subgroup for $k \leq n/2$, then $\mathrm{cd}(\Gamma) \leq n-k$.
Also, there exists a compact collection of $n-k$-dimensional affine subspaces where $\Gamma$ acts. The techniques employed here mostly stem from the coarse geometry theory.
Canary and Tsouvalis previously proved the same result using the powerful method of Bestvina and Mess for word hyperbolic groups; however, our approach differs in that  our method projects the holonomy cover to a stable affine subspace, and we plan to generalize to relative Anosov groups.

		
	\end{abstract}
	\maketitle
	\tableofcontents

\part{Hyperbolic bundles and  Anosov representations} \label{part2}  

\section{Introduction}


\subsection{Notation}
Let $\Aff(\mathds{A}^n)$ denote the group of affine transformations of 
an affine space $\mathds{A}^n$
whose elements are of the form: 
\[  x \mapsto A x + \bv   \]
for a vector $\bv \in \bR^{n}$ and $A \in \GL(n, \bR)$. 
Let $\mathcal{L}: \Aff(\mathds{A}^n) \ra \GL(n, \bR)$ denote 
the map sending the affine transformations to their linear parts.

 Let $\Gamma$ be an affine group acting properly and freely on $\mathds{A}^n$. 
 We assume that the Eilenberg-Mac Lane space $K(\Gamma, 1)$ for dimension one is realized as a finite complex, which is a fairly mild assumption in our field. 
 Let $\mathds{A}^n/\Gamma$ be a manifold equipped with a complete Riemannian metric structure. 

Let $M$ be a compact submanifold (with or without boundary) in $\mathds{A}^n/\Gamma$, where $\Gamma$ acts properly discontinuously and freely on $\mathds{A}^n$. We assume that the inclusion map $M \hookrightarrow \mathds{A}^n/\Gamma$ induces a surjection on fundamental groups. The inverse image $\hat{M}$ of $M$ in $\mathds{A}^n$ is connected, and the deck transformation group $\Gamma_M$ of $\hat{M}$ is isomorphic to $\Gamma$. In these cases, we refer to $M$, the map $M \hookrightarrow \mathds{A}^n/\Gamma$, $\hat{M}$, the projection $p_{\hat{M}}: \hat{M} \to M$, and $\Gamma_M$ as an \emph{fundamental group surjective (FS) submanifold}, an \emph{FS map}, an \emph{FS cover}, an \emph{FS covering map}, and an \emph{FS deck transformation group}, respectively.


A \emph{compact core} of a manifold is a compact submanifold homotopy equivalent to the ambient manifold. 
When a manifold is homotopy equivalent to a finite complex $K$, for $n=3$, Scott and Tucker \cite{ST89} proved the existence of a compact core, while for $n - \dim K \geq 3$, Stallings \cite{Stallings65} established this fact. Unfortunately, compact cores do not always exist (see Venema \cite{Venema80}). Although compact cores are preferable for our arguments, we sometimes need to resort to FS submanifolds.

Let $d_E$ denote the chosen standard Euclidean metric on $\mathds{A}^n$, fixed for this paper. Now, $\mathds{A}^n$ has an induced complete $\Gamma$-equivariant Riemannian metric from $N := \mathds{A}^n/\Gamma$, denoted by $\dN$, where $\tilde{N} = \mathds{A}^n$.

We will assume that $\partial M$ is convex in this paper. Let $d_M$ denote the path metric induced from a Riemannian metric on $\mathds{A}^n/\Gamma$, and let $d_{\hat M}$ denote the path metric on $\hat{M}$ induced from it.

  



\begin{itemize} 
\item Let $\Uu M$ denote the space of direction vectors on $M$, i.e., 
the space of elements of the form 
$(x, \vec{v})$ for $x\in M$ and a unit vector $\vec{v} \in T_x M - \{O\}$ 
with the projection $\pi_{\Uu M}: \Uu M \ra M$. 
\item Let $\Uu \hat M$ denote the space of direction vectors on $\hat M$
covering $\Uu M$ with the deck transformation group $\Gamma_M$.  
Let $\pi_{\Uu \hat M} : \Uu \hat M \ra \hat M$ denote the projection.
 \item  Let $d_{\Uu M}$ denote the path metric on $\Uu M$ obtained from
  the natural Riemannian metric
on the unit tangent bundle $\Uu M$ of $M$ 
obtained from the Riemannian metric of $M$
as a sphere bundle.   We will use 
$d_{\Uu \hat M}$ to denote the induced path metric on $\Uu \hat M$. 
\end{itemize} 


  
  A {\em complete isometric geodesic} in $\hat M$ is a geodesic that is an isometry of 
  $\bR$ into $\hat M$ equipped with $M$. 
 Some but not all Riemannian geodesics are isometries. 
Actually, in the geodesic metric spaces, these are identical to the notion of geodesics.
Since we are working on Riemannian metric spaces, we use adjectives.  
  A {\em complete isometric geodesic} in $M$ is a geodesic that lifts to 
  a complete isometric geodesic in $\hat M$. 
A {\em ray} in $\hat M$ is an isometric geodesic defined on  
$[0, \infty)$. 
In this paper, the geodesic is always an isometric geodesic. 

  
%
  

We consider the subspace of $\Uu M$ where complete isometric geodesics pass. We denote this subspace by $\Uc M$, and call it the \emph{complete isometric geodesic unit-tangent bundle}. The inverse image in $\Uu \hat M$ is a closed set denoted by $\Uc \hat M$. Since a limit of a sequence of isometric geodesics is an isometric geodesic, $\Uc M$ is compact and $\Uc \hat M$ is closed and locally compact. However, $\pi_{\Uu M}(\Uc M)$ may be a proper subset of $M$ for the projection $\pi_{\Uu M}: \Uu M \to M$, and 
$\pi_{\Uu \hat M}(\Uc \hat M)$ may be a proper subset of $\hat M$.

Let $\rho: \Gamma_M \ra \GL(n, \bR)$ be a homomorphism. 
There is an action of $\Gamma_M$ 
on $\Uc \hat M \times \bR^n$ given by 
\[ \gamma\cdot(x, \vec{v})  = (\gamma(x), \rho(\gamma)(\vec{v})) 
\hbox{ for } x \in \Uc\hat M, \vec{v} \in \bR^n.\]
Let $\bR^n_{\rho}$ denote the bundle over $\Uc M$
 given by taking the quotient $\Uc \hat M \times \bR^n$ by 
 $\Gamma_M$. 

Recall that there is a geodesic flow $\phi$ on $\Uu M$ restricting to one on $\Uc M$, also denoted by $\phi$. This lifts to a flow, still denoted by $\phi$, on $\Uc \hat M$. Hence, there is a flow $\Phi$ on $\Uc \hat M \times \mathbb{R}^n$, which descends to a flow, also denoted by $\Phi$, on $\mathbb{R}^n_{\rho}$.


\subsection{Main result} 
In this paper, we establish connections between partial hyperbolicity in the context of dynamical systems and P-Anosov properties. The significance and motivation of the main result are detailed in Part 2.

\red{Unlike the abstract dominated splitting discussed in many previous articles such as \cite{BPS} and \cite{KP22}, here we require actual splitting of smooth bundles and smooth flows, along with the presence of expanding and contracting properties in addition to the domination property.}



We  need the following objects
since we need to work on manifolds for Part 2. 
\begin{definition}[Partial hyperbolicity]\label{defn:phyp} 
	Suppose that $M$ is a compact  Riemannian manifold with convex boundary.
	Let $\hat M$ be a regular cover of $M$ with the deck transformation group $G_M$. 
	A representation $\rho: G_M \ra \GL(n, \bR)$ is {\em partially hyperbolic in a bundle sense} if 
	the following hold for a Riemannian metric on $M$: 
	\begin{enumerate} 
		\item[(i)] 
		There exist nontrivial $C^0$-subbundles $\bV_+, \bV_0$, and $\bV_-$ in $\bR^n_{\rho}$ 
		with fibers \[\bV_+(x), \bV_0(x), \bV_-(x) \subset \bV(x)\] for the fiber $\bV(x)$ of $\bV$ over each point $x$ of $\Uc M$
		invariant under the flow $\Phi_t$ of $\bV$ over 
		the geodesic flow $\phi_t$ of $\Uc M$. 
		\item[(ii)] $\bV_+(x), \bV_0(x)$ and $\bV_-(x) $ for 
		each $x \in \Uc M$ are independent subspaces, and their sum equals $\bV(x)$.
		\item[(iii)] For any fiberwise metric on $\bR^n_{\rho}$ over $\Uc M$, 
		the lifted action of $\Phi_{t}$ on $\bV_{+}$ (resp. $\bV_{-}$) is dilating (resp. contracting). 
		That is, for coefficients $A> 0, a> 0$, $A' > 0, a'>0$\/: 
		\begin{enumerate} 
			\item $\llrrV{\Phi_{-t}(\vv)}_{\bR^n_\rho, \phi_{-t}(m)} \leq A \exp(-a t)\llrrV{\vv}_{\bR^n_\rho, m}$ for $\vv \in \bV_+(m)$ as $t \ra \infty$. 
			\item $\llrrV{\Phi_{t}(\vv)}_{\bR^n_\rho, {\phi_{t}(m)}} \leq A \exp(-a t)\llrrV{\vv}_{\bR^n_\rho, m}$ for $\vv \in \bV_-(m))$ as $t \ra \infty$. 
			\item (Dominance properties) 
			\begin{multline} \label{eqn:dominance} 
				\frac{\llrrV{\Phi_{t}(\vw)}_{\bR^n_\rho,\phi_{t}(m)}}{\llrrV{\Phi_{t}(\vv)}_{\bR^n_\rho, \phi_{t}(m)}} 
				\leq A' \exp(-a't)\frac{\llrrV{\vw}_{\bR^n_\rho, m}}{\llrrV{\vv}_{\bR^n_\rho, m}} \\
				\hbox{ for } \vv \in \bV_+(m), \vw \in \bV_0(m)
				\hbox{ as } t \ra \infty, \\
				 \hbox{ or for }
				 \vv \in \bV_0(m), \vw \in \bV_-(m)
				 \hbox{ as } t \ra \infty.
			\end{multline} 
		\end{enumerate} 
	\end{enumerate}
	We additionally require that $\dim \bV_+ = \dim \bV_- \geq 1$.  
	Here $\dim \bV_+$ is the {\em partial hyperbolicity index} of $\rho$
	for the associated bundle decomposition $\bV_+, \bV_0, \bV_-$ 
	where $\dim \bV_+$ is maximal
	and $\dim \bV_+ = \dim \bV_- < n/2$, so that $\dim \bV_0 \geq 1$. 
 Furthermore,	$\bV_0$ is said to be the {\em neutral subbundle} of $\bV$. 
	(See Definition 1.5 of \cite{CP} or \cite{BLM}.) 
\end{definition} 
The definition does depend on the Riemannian metric but not on $\llrrV{\cdot}_{\bR^n_{\rho}}$. However, if $\Gamma_M$ is word hyperbolic, then we can replace $\Uc \hat M$ with a Gromov flow space which depends only on $\Gamma_M$ up to quasi-isometry. (See Theorem \ref{thm:identify} and Gromov \cite{Gromov87} or Mineyev \cite{Mineyev}.)

We say that a complete affine manifold $N$ has a \emph{partially hyperbolic linear holonomy group} if
\begin{itemize}
\item there exists an FS submanifold $M$ in $N$ with convex boundary for a Riemannian metric on $N$, and
\item the linear part $\rho: \pi_1(N) = \Gamma_{\hat M} \to \GL(n, \mathbb{R})$ of an affine holonomy homomorphism $\rho'$ is a partially hyperbolic representation in a bundle sense.
\end{itemize}

We say that a complete affine manifold $N$ has a \emph{partially hyperbolic linear holonomy group} if
\begin{itemize}
\item there exists an FS submanifold $M$ in $N$ with convex boundary for a Riemannian metric on $N$, and
\item the linear part $\rho: \pi_1(N) \cong \Gamma_{\hat{M}} \to \GL(n, \mathbb{R})$ of an affine holonomy homomorphism $\rho'$ is partially hyperbolic in a bundle sense.
\end{itemize}




Since partial hyperbolicity is equivalent to the Anosov property by Theorem \ref{thm:main2}, $\pi_1(N)$ is always hyperbolic. By Theorem \ref{thm:identify}, there exists a quasi-isometry from $\Uc \hat M$ to the Gromov flow space $\partial^{(2)}_\infty M \times \mathbb{R}$ preserving the geodesic length parameter. Hence, the discussion reduces to one on the Gromov flow space, which is independent of the choice of FS submanifolds.

Let $\log \lambda_1, \dots, \log \lambda_n$ denote the logarithmic eigenvalue functions defined on
the maximal abelian Lie algebra of $\mathfrak{gl}(n, \bR)$. 
For roots 
\[\theta=\left\{\log \lambda_{i_{1}}-\log \lambda_{i_{1}+1}, \ldots, \log \lambda_{i_{m}}- \log \lambda_{i_{m}+1}\right\}, 
1 \leq i_{1}<\cdots<i_{m} \leq n-1,\] 
of $\mathrm{GL}(n, \bR)$, 
the parabolic
subgroup $P_{\theta}$ (resp. $\left.P_{\theta}^{-}\right)$ is 
the subgroup of block upper 
(resp. lower) triangular matrices in $\mathrm{GL}(n,\bR)$ with square diagonal blocks 
of sizes $i_{1}, i_{2}-i_{1}, \ldots, i_{m}-i_{m-1}, n-i_{m}$ respectively.

\begin{theorem}\label{thm:main2} 
Let $N$ be a complete affine $n$-manifold where 
 $K(\pi_1(N), 1)$ is realized by a finite complex. 
	Let $\rho: \pi_1(N)\ra \GL(n, \bR)$ is a linear part of an affine holonomy homomorphism $\rho'$. 

Then $\rho$ is $P$-Anosov for 
	any parabolic subgroup $P$ of $\GL(n, \bR)$ with 
	$P = P_\theta$ for $\theta$ containing $\log \lambda_k-\log \lambda_{k+1}, k \leq n/2$ 
	if and only if $\rho$ is partially hyperbolic of index $k$. 
	In these cases $k < n/2$. 
\end{theorem}

The hyperbolic flow was already discovered by Goldman and Labourie \cite{GL12} for $3$-dimensional Margulis space-times, and Ghosh \cite{Ghosh17}, Ghosh-Treib \cite{GhoshTreib21}, and Danciger-Zhang \cite{DZ2019} prove these properties for affine actions on higher dimensional affine spaces. However, the linear parts are always assumed to be in the orthogonal groups where the properties are more or less immediate.

\red{The significance of this result is that we work with the full general linear group $\GL(n, \mathbb{R})$ instead of orthogonal groups. Also, Anosov representations correspond to only dominated representations and not actual partially hyperbolic flows as here. Promoting the dominated representation to Anosov ones is trivial for some Lie groups but not for general linear groups. }



We aim to prove in Part 2 the following theorem worked out with Kapovich. 
\begin{theorem}[Theorem \ref{thm:main-p2}]\label{thm:mainofpart2} 
Let $N$ be a complete affine manifold  for $n \geq 3$ with the finitely presented fundamental group. 
Suppose that $N$ has a partially hyperbolic linear holonomy group with index $k, k < n/2$,  and $K(\pi_1(N), 1)$ is realized by a finite complex.  

Then the cohomological dimension $\mathrm{cd}(\pi_1(N))$ is $\leq n-k$ for the partial hyperbolicity index $k$ of $\rho$. 
\end{theorem} 
\red{This result coincides with Canary-Tsouvalis \cite{CT2020} on linear Anosov representations for $\SL(n, \mathbb{R})$ up to minor additional arguments, while our paper employs a significantly different method of homotopying the manifold into an affine subspace. However, we aim to generalize most of these methods and results to complete affine manifolds with relatively Anosov holonomy representations, a field that is still developing. See Kapovich-Leeb \cite{kapovich2023relativizing}, Zhu \cite{Zhu21,Zhu23}, and Zhu-Zimmer \cite{zhu2022p}. The approach of Canary-Tsouvalis \cite{CT2020}, which builds upon the work of Bestvina-Mess \cite{BM91}, essentially does not extend straightforwardly to the context of relativizing.}


\red{The reason for dividing into two parts is that the main techniques are distinct, and it seems distracting to switch gears in the middle of the paper. We believe this approach is more reader-friendly. The first part focuses on partially hyperbolic flows, primarily working with singular values and flows. The second part delves into coarse geometry to a greater extent.}

\red{Additionally, the techniques involved--- Anosov representations, partially hyperbolic flows, and coarse geometry --- may not be familiar to traditional geometric topologists as they are rapidly developing fields. However, the results used are grounded in well-established papers.}

This paper aimed to address the conjecture that a closed complete affine manifold cannot have word hyperbolic fundamental groups. One approach considered is using Osedelec-type theorems (See Chapter 8 of \cite{KH95}). However, generalizing our results to measurable settings is currently hindered by numerous technical issues. Discussions on these matters took place with researchers such as Nicholas Tholozan at BIRS, Banff, during the winter of 2020.





\subsection{History of the field relevant to this result} 
Originally, Milnor posed the question of whether a free affine group of rank $\geq 2$ could act freely and properly discontinuously on affine space. Margulis provided an affirmative answer by constructing examples in the affine $3$-space, where the linear parts belong to $\SO(2, 1)$. Drumm and Goldman subsequently explained this result using fundamental domains in a series of papers (see \cite{DDGS} for a survey). Sullivan \cite{Sullivan76} observed that the linear part of each affine group element must have $1$ as an eigenvalue and proved the Chern conjecture for complete affine manifolds \cite{KS75}. Klingler \cite{Klingler2017} later extended this to closed special affine manifolds.

Further developments include Ghosh \cite{Ghosh17,Ghosh18,ghosh2018avatars,Ghosh23}, Ghosh-Treib \cite{GhoshTreib21}, and Smilga \cite{Smilga14,Smilga16,Smilga18,Smilga22,Smilga22ii}, who generalized these results and produced higher-dimensional examples of proper and free affine actions. Abels, Margulis, and Soifer also contributed to related topics such as the Auslander conjecture (see Abels \cite{Abels01} for a survey). Danciger and Zhang \cite{DZ2019} investigated affine deformations of Hitchin representations, showing they do not exist. Danciger, Kassel, and Gu\'ertaud \cite{DGK2020} provided examples of closed lower-dimensional manifold fundamental groups acting properly and freely on affine spaces.

While these works are related to the themes of this paper, their specific aims differ. This summary only scratches the surface of this rapidly evolving field.

\subsection{Outline of Part 1}

In Section \ref{sec:preliminary-p1}, we will introduce  affine structures and 
some basic facts on $\delta$-hyperbolic metric spaces. 
Suppose that $\Gamma_M$ is word-hyperbolic.
Then $\hat M$ has a Gromov hyperbolic Riemannian metric
by the Svarc-Milnor theorem. 
We discuss the ideal boundary of $\hat M$, identifiable with the Gromov boundary, and the complete isometric geodesics in $\hat M$. 
We relate the space of complete isometric geodesics to
the Gromov flow space. 

In Section \ref{sec:phyp}, 
we provide a review of the definition of the $P$-Anosov property of 
Guichard-Wienhard \cite{GW12}. For convenience for reductive groups, we will use 
the approaches of G\'ueritaud-Guichard-Kassel-Wienhard \cite{GGKW17ii} and Bochi-Potrie-Sambarino \cite{BPS}
using singular values of the linear holonomy group. (The theory of 
Kapovich, Leeb, and Porti \cite{KL18}, \cite{KLP18iii} is equivalent to these, but it is not so
explicit for our purposes here.)
We give various associated definitions. 

In Section \ref{sec:PAnosov}, we have the main arguments. 
In Section \ref{sub:kconvexity}, we will relate the partial hyperbolic property in the singular value sense to that in the bundle sense. 
In Section \ref{sub:converse}, we show that the linear part of 
the holonomy representation of a complete affine manifold 
is $P$-Anosov if and only if it is partially hyperbolic. 
First, we do this when the linear holonomy group
has a reductive Zariski closure. 
This involves the fact that $1$ has to be an eigenvalue of each element of 
the linear holonomy groups of complete affine manifolds. We relate it
to singular values using the clever ideas of Danciger and Stecker,
the spectrum theory of Breulliard-Sert \cite{BS2021}, Benoist \cite{Benoist97}, \cite{Benoist98}, and the work of Potrie-Kassel \cite{KP22}.
We could generalize to the cases of nonreductive linear holonomy groups using 
the stability of the Anosov and partial hyperbolic properties. 
This will prove Theorem  \ref{thm:main2}. 

\subsection{Acknowledgments}
We thank Michael Kapovich for various help with geometric group theory and coarse geometry. This article began with some discussions with Michael Kapovich during the conference honoring the 60th birthday of William Goldman at the University of Maryland, College Park, in 2016. 
Important suggestions for proofs of our main Theorems \ref{thm:main2} and \ref{thm:mainofpart2} were made by Michael Kapovich.
We thank Jeffrey Danciger and Florian Stecker for help with $P$-Anosov properties which yielded the main idea for proof of Theorem \ref{thm:main2}. 
%
%
We also thank Herbert Abels, Richard Canary, Virginie Charette, Todd Drumm,
William Goldman, Fran\c{c}ois Gu\'eritaud, Fanny Kassel, 
Andr\'es Sambarino, Nicholas Tholozan, and Konstantinos Tsouvalas for various discussions helpful for this paper. 
We also thank BIRS, Banff, Canada, and KIAS, Seoul, where some of this work was done.
\red{This paper is a joint work with Kapovich in the beginning and Danciger and Stecker later. Unfortunately, they did not wish to join as coauthors for various reasons. }

\section{ Preliminary} \label{sec:preliminary-p1} 


\red{We recall some standard well-known facts in this section to set the notations and so on. 
Also, we need a slight modification of the theories for our purposes. Most of the backgrounds on geometric groups and coarse geometry are from the comprehensive textbook by 
Dru\c{t}u and Kapovich \cite{DK2018}.}

\subsection{Convergences of geodesics} \label{sub:geoconv}  
Let $M$ be a manifold of dimension $n$ with a regular cover $\hat{M}$ and the deck transformation group $\Gamma_M$. We assume that $M$ has a Riemannian metric with convex boundary.

We say that a sequence of rays (resp. complete isometric geodesics) $l_i$ \emph{converges} to a ray (resp. complete geodesic) $l$ (denoted as $l_i \to l$) if $l_i(t) \to l(t)$ for each $t \in [0, \infty)$ (resp. $t \in \mathbb{R}$).

By the Arzel\`a-Ascoli theorem and the properness of $\hat{M}$, any sequence of rays (resp. complete isometric geodesics) converges to a ray (resp. complete isometric geodesic) if the sequence of pairs of their $0$-points and directions at the $0$-points converges in $\Uu \hat{M}$. (See Section 11.11 of \cite{DK2018}.)



\subsection{$(G,X)$-structures and affine manifolds}
Let $G$ be a Lie group acting transitively and faithfully on 
a space $X$. 
A {\em $(G, X)$-structure} on a manifold $N$ is a maximal atlas of charts to $X$ so that 
transition maps are in $G$. 
This is equivalent to
$M$ having a pair $(\dev, h)$ where 
$\hat M$  is a regular cover of $M$ with a deck transformation group $\Gamma_M$
such that
\begin{itemize} 
	\item $h:\Gamma_M \ra G$ is a homomorphism, called a {\em holonomy homomorphism}, and 
	\item $\dev: \hat M \ra X$ is an immersion,  
	called a {\em developing map}, satisfying 
	\[\dev \circ \gamma = h(\gamma)\circ \dev \]
 for each deck transformation  $\gamma \in \Gamma_M$.
\end{itemize}


We say that an $n$-manifold $M$ 
with a $(G, X)$-structure  is {\em complete} if $\dev: \hat M \ra X$ is a diffeomorphism. 
If $M$ is complete, 
we have a diffeomorphism 
$M \ra X/h(\Gamma_M)$, and $h(\Gamma_M)$ acts properly discontinuously and freely on $X$. Complete $(G, X)$-structures on $M$ are classified 
by the conjugacy classes of representations $\Gamma_M \ra G$ with properly discontinuous and free actions on $X$. 
(See Section 3.4 of \cite{Thurston97}.)

Note that the completeness and compactness of $M$ have 
no relation for some $(G,X)$-structures 
unless $G$ is in the isometry group of $X$ with a Riemannian metric.
(The Hopf-Rinow lemma may fail.)

Now, we go over to the affine structures. 
Let $\mathds{A}^n$ be a complete affine space. 
	Let $\Aff(\mathds{A}^n)$ denote the group of affine transformations of $\mathds{A}^n$
	whose elements are of the form: 
	\[  x \mapsto A x + \bv   \]
	for a vector $\bv \in \bR^{n}$ and $A \in \GL(n, \bR)$. 
	Let $\mathcal{L}: \Aff(\mathds{A}^n) \ra \GL(n, \bR)$ 
	denote map sending each element of $\Aff(\mathds{A}^n)$ to its linear part in $\GL(n, \bR)$.

		An affine $n$-manifold is an $n$-manifold with  an $(\Aff(\mathds{A}^n), \mathds{A}^n)$-structure. 
		A {\em complete affine $n$-manifold}
			is an $n$-manifold $M$ of the form $\mathds{A}^n/\Gamma$ for $\Gamma \subset \Aff(\mathds{A}^n)$
			acting properly discontinuously and freely on $\mathds{A}^n$.

\subsection{Gromov hyperbolic spaces} \label{sub:Gromov} 


 Let us recall some definitions: 
A metric space is {\em geodesic} if every pair of points are connected by a geodesic, i.e., 
a path that is an isometry from an interval to the metric space.  
We will only work with geodesic metric spaces.
 Let $(X, d_X)$ and $(Y, d_Y)$ be metric spaces. A map $f: X \ra Y$ is called 
{\em  $(L, C')$-coarse Lipschitz} for $L, C > 0$ if 
 \[d_Y(f(x), f(x') \leq L d_X(x, x') + C' \hbox{ for all } x, x' \in X.\] 
 A map $f: X \ra Y$ is an {\em $(L, C)$-quasi-isomeric embedding} if 
 $f$ satisfies 
 \[L^{-1} d_X(x, x') - C \leq d_Y(f(x), f(x')) \leq L d_X(x, x')  + C \hbox{ for all }x, x'\in X.\] 
 A map $\bar f: Y\ra X$ is a {\em $C$-coarse inverse} of $f$ 
 for $C > 0$ if 
\[d_X(\bar f \circ f, \Idd_X) \leq C, \hbox{ and }
d_Y(f\circ \bar f, \Idd_Y)\leq C.\]  
A map $f:X \ra Y$ between metric spaces is called 
 a {\em quasi-isometry} if it is a coarse Lipschitz map and admits a 
 coarse Lipschitz coarse inverse map.
(Note that these are not necessarily continuous as we follow 
Dru\c{t}u-Kapovich \cite{DK2018}.) 

\begin{lemma} \label{lem:piUUM} 
	$\pi_{\Uu \hat M}$ is a quasi-isometry.
	\end{lemma} 
\begin{proof}
	Each fiber is uniformly bounded under $d_{\Uu \hat M}$.  
	\end{proof}

A metric space $(X, d)$ is {\em proper} if every metric ball has a compact closure. 
In other words, $d_p(\cdot) = d(p, \cdot)$ is a proper function $X \ra \bR$. Clearly,
a complete Riemannian manifold is a proper metric space.
We denote by $B^d(x, R)$ the ball of radius $< R$ in $X$ with the center $x \in X$. 

Let $X$ be a geodesic metric space with metric $d$. 
A {\em geodesic triangle $T$} is a concatenation of three geodesics $\tau_1, \tau_2, \tau_3$ where the indices are modulo $3$ ones. 
The {\em thinness radius} of a geodesic triangle $T$ is the number 
\[ \delta(T) := \max_{j=1,2,3} \left( \sup_{p \in \tau_j} d(p, \tau_{j+1}\cup \tau_{j+2}) \right). \]

A geodesic metric space $X$ is called {\em $\delta$-hyperbolic} in the sense of Rips if 
every geodesic triangle $T$ in $X$ is $\delta$-thin. 

By Corollary 11.29 of \cite{DK2018} as proved by Section 6.3C of 
Gromov \cite{Gromov87}, 
the Gromov hyperbolicity is equivalent to the Rips hyperbolicity
for geodesic metric spaces.
We will use these concepts interchangeably.

We will assume that $\Gamma_M$ is word-hyperbolic for our discussions below. 
This is equivalent to assuming that 
$\hat M$ is Gromov hyperbolic by the Svarc-Milnor Lemma 
(see \cite{Milnor} and \cite{Svarc}).

Two rays $\rho_1$ and $\rho_2$ of $X$ are {\em equivalent} if 
$t\mapsto d(\rho_1(t), \rho_2(t))$ is bounded.  
This condition is equivalent to the one that their Hausdorff distance under $d$ is bounded. 
Assume that $X$ is $\delta$-hyperbolic. $X$ has a well-defined ideal boundary 
$\partial_\infty X$ as in Definition 3.78 of \cite{DK2018}, 
i.e., the space of equivalence classes of geodesic rays. 
 (See \cite{CDP90} and \cite{GdH90}.) 
 We denote by $\partial^{(2)}_\infty X = \partial_\infty X \times \partial_\infty X - \Delta$ 
 where $\Delta$ is the diagonal set. 

The constant $C$ below is called a {\em quasi-geodesic} constant. 
 \begin{lemma} \label{lem:twogeo} 
	Let $M$ be a compact manifold with 
	the induced path metric $d_{\hat M}$ on $\hat M$
	from a Riemannian metric of $M$.
	Let $\Gamma_M$ be the deck transformation group of $\hat M \ra M$.  
	Suppose that $\Gamma_M$ is word-hyperbolic. 

Then there exists a constant $C >0$ so that
for	every pair of
	complete isometric geodesics $l_1$ and $l_2$ in $\hat M$
with the identical set of endpoints in $\partial_\infty \hat M$, the Hausdorff distance between 
$l_1$ and $l_2$ under $d$ is bounded above by $C$. 
\end{lemma} 
\begin{proof}
	Since $l_1$ and $l_2$ are both $(1, 0)$-quasi-geodesics, 
	this follows from Proposition 3.1 of Chapter 3 of \cite{CDP90}. 
\end{proof}

 	We put on $\hat M \cup \partial_\infty \hat M$ the 
 shadow topology, which is a first-countable Hausdorff space
 by Lemma 11.76 of \cite{DK2018}. It is also compact according to Section 11.11 of \cite{DK2018}. 
 Since it is first countable, 
 we do not need to consider nets on the space but only sequences
 to understand the continuity of the real-valued functions. 

We denote by $\partial_G X$ the Gromov boundary of 
a $\delta$-hyperbolic geodesic metric space $X$. (See Section 11.12 of \cite{DK2018}.)
 By Theorem 11.104 of \cite{DK2018}, 
 there is a homeomorphism $h: \partial_\infty X
 \ra \partial_G X$ given 
 by sending the equivalence class $[\rho]$ of a ray $\rho$
 to the equivalence class of $\{\rho(n)\}_{n \in \bZ_+}$. 
 We will identify these two spaces using this map. 

Let $\hat \Gamma_M$ denote the Cayley graph of $\Gamma_M$. 
The Gromov boundary $\partial_G \hat M$ of $\hat M$ and 
can be identified 
with the boundary of $\hat \Gamma_M$ with the word metric
by Theorem 11.108 of \cite{DK2018}. 

For any Gromov hyperbolic geodesic metric space $X$, 
the set of $GX$ of complete isometric geodesics has a metric: 
for $g, h \in GX$, we 
define the metric $d_{GX}$ given by 
\begin{equation}\label{eqn:dGX} 
 d_{GX}(g, h) := \int^\infty_{-\infty} d(g(t), h(t)) 2^{-|t|} dt.    
\end{equation}
(See Gromov \cite{Gromov87}.)
%
%
Let $GX$ have this metric topology. 
$GX$ is clearly locally closed since the limit of 
a sequence of isometries from $\bR$ is an isometry from $\bR$. 
(See Section 11.11 of \cite{DK2018}.)

\begin{lemma} \label{lem:identify} 
	Suppose that $\hat M$ is Gromov hyperbolic. 
	Then there is a quasi-isometric homeomorphism $\mathcal{F}:\Uc \hat M \ra G\hat M$
	by taking a unit vector $\vec{u}$ at $x$ to a complete isometric geodesic 
	$\bR \ra \hat M$ passing $x$ tangent to $\vec{u}$. 
	\end{lemma} 
\begin{proof}
There is a map $G\hat{M} \to \Uc \hat{M}$ given by sending the complete isometric geodesic $g: \mathbb{R} \to \hat{M}$ to $(g(0), \vec{v}_0)$ where $\vec{v}_0$ is the unit tangent vector at $g(0)$. This map is clearly bijective. In Section III of \cite{Matheus90}, the map $G\hat{M} \to \hat{M}$ defined by $g \mapsto g(0)$ is shown to be a quasi-isometry. Therefore, post-composing this map with a lift that sends each $g(0)$ to a vector in the fiber of $\Uc \hat{M}$ over $g(0)$ results in a quasi-isometry. For a sequence $\{g_i\}$ of geodesics in $\hat{M}$, if $d_{GX}(g, g_i) \to 0$, then $(g_i(0), g_i'(0)) \to (g(0), g'(0))$ since otherwise, we would obtain a positive lower bound for the integral \eqref{eqn:dGX}.

The inverse map $\Uc \hat{M} \to G\hat{M}$ is also continuous due to the continuity of the exponential map and considerations involving \eqref{eqn:dGX}, where $d_X(g(t), h(t))$ grows sublinearly. The continuity of integral values under $g$ and $h$ can be shown by cutting off for $|t| > N$.

	
	\end{proof}

%


\subsection{Flow space $G\hat M$ quasi-isometric to $\partial_\infty^{(2)}\hat M \times \bR$}

We assume that $\Gamma_M$ is word-hyperbolic and 
denote by $F\Gamma_M$ the Gromov 
geodesic flow space obtained by the following proposition. 
For a group $\Gamma$, we let $\hat \Gamma$ denote its Cayley graph.
We denote $\partial_\infty^{(2)}\hat \Gamma:=\partial_\infty\hat \Gamma\times\partial_\infty\hat \Gamma - \{(t,t)|t\in\partial_\infty\hat \Gamma\}$. 

\begin{theorem}[Theorem 8.3.C of Gromov \cite{Gromov83}, Theorem 60 of Mineyev \cite{Mineyev}] 
	\label{thm:Mineyev} Let $\Gamma$ be a finitely generated word hyperbolic group. 

Then there exists a proper hyperbolic metric space 
	$F\Gamma$ with the following properties\/{\em :}
	\begin{itemize} 
		\item[(i)] $\Gamma\times\bR\times\bZ/2\bZ$ acts on $F\Gamma$.
		\item[(ii)] The $\Gamma\times \bZ/2\bZ$-action is isometric.
		\item[(iii)] Every orbit $\Gamma\ra F\Gamma$ is a quasi-isometry. 
		\item[(iv)] The $\bR$-action is free, and every orbit $\bR \ra F\Gamma$ is a quasi-isometric embedding. 
		The induced map $F\Gamma/\bR \ra \partial_\infty^{(2)}\hat \Gamma$ is a homeomorphism.
	\end{itemize}
\end{theorem}
We shall say that $F \Gamma$ is the {\em flow space} of $\Gamma$. 
In fact, $F \Gamma$ is unique up to a $\Gamma \times \bZ/2\bZ$-equivariant quasi-isometry sending $\bR$-orbits to $\bR$-orbits. 
We shall denote by $\phi_t$ the $\bR$-action on $F \Gamma$ and by 
$(\tau_+, \tau_-):F\Gamma \ra F\Gamma/\bR \cong \partial_\infty^{(2)}\hat \Gamma$
the maps associating to an element of $F\Gamma$ 
the endpoint of its $\bR$-orbit.
Gromov identifies $F\Gamma$ with 
$\partial_\infty^{(2)} \Gamma \times \bR$ with a certain metric 
called the Gromov metric. 

\begin{proposition} \label{prop:Uc} 
	Let $M$ be a compact manifold with a covering map $\hat M \ra M$
	with a deck transformation group $\Gamma_M$. 
	Suppose that $\Gamma_M$ is word-hyperbolic.

	Then there is 
	a quasi-isometric surjective map 
	$\mathcal{E}: G\hat M \ra \partial^{(2)}_\infty \hat M \times \bR$ with compact fibers. 
	The map obtained from composing with 
	a projection to the first factor is a map given by taking the endpoints of complete isometric geodesics. 
\end{proposition} 
\begin{proof} 
	Choose a basepoint $x_0$ in $\hat M$. 
	For each complete isometric geodesic $g: \bR \ra \hat M$, 
	let $g_{x_0}$ denote the projection point on $g(\bR)$ that is of 
	the minimal distance from $x_0$
	and let $\partial_+ g, \partial_-g \in \partial_\infty \hat M$ denote
	the forward and backward endpoints,  
	and let $t(g) \in \bR$ denote $\pm d(g(0), g_{x_0})$ where we use $+$ 
	if $g(0)$ is ahead of $g_{x_0}$.  
	
	We define 
	$\mathcal{E}(g) = (\partial_+ g, \partial_- g, t(g))$. 
	The surjectivity follows from Proposition 2.1 of Chapter 2 of 
	\cite{CDP90} since $t$ is an isometry on each complete isometric  
	geodesic. 	The compactness of the fiber is implied by Lemma \ref{lem:twogeo}. 
	
	

	Champetier \cite{Champetier94} constructs the space 
	$G\hat \Gamma_M$ from $\hat \Gamma_M$ quasi-isometric to $\Gamma_M$ 
	following Gromov. 
Recall $\partial_\infty \hat{M} = \partial_\infty \hat{\Gamma}_M$. Using the orbit map $\Gamma_M \to \hat{M}$ sending $\gamma$ to $\gamma(x_0)$, we extend the map to $\hat{\Gamma}_M \to \hat{M}$ by sending each edge to a shortest geodesic.

Let $I: G\hat{\Gamma}_M \to G\hat{M}$ be a map that sends a geodesic $g$ in $\hat{\Gamma}_M$ to a geodesic $g'$ in $\hat{M}$ with the same endpoints in $\partial_\infty \hat{M}$, and maps $g(0)$ to the nearest point on the image of $g'$.

The map $I': G\hat{M} \to G\hat{\Gamma}_M$ can be defined by taking a geodesic $g$ in $\hat{M}$ to a geodesic $g'$ in $\hat{\Gamma}_M$ in a similar manner, with $g(0)$ going to one of the elements of $\Gamma_M(x_0)$ nearest to it.

By Theorem 11.72 of \cite{DK2018}, which states that every geodesic in $\hat{\Gamma}_M$ is a quasi-geodesic in $\hat{M}$ and vice versa, and considering that every geodesic in $\hat{M}$ is uniformly bounded away from one in $\hat{\Gamma}_M$ in the Hausdorff distance, it follows that $G\hat{M}$ is quasi-isometric to $G\hat{\Gamma}_M$ via the maps $I$ and $I'$, using \eqref{eqn:dGX}.

	
	We define 
$\mathcal{E}': G\hat \Gamma_M \ra \partial_\infty^{(2)} \hat \Gamma_M \times \bR$ 
	as we did for $\mathcal{E}$ above. 
	By Proposition 4.8 and (4.3) of \cite{Champetier94}, 
	$\mathcal{E}'$ sends 
	$G\hat \Gamma_M$ to 
	$\partial_\infty^{(2)} \hat \Gamma_M \times \bR
	= \partial_\infty^{(2)}\hat M\times \bR$ as a quasi-isometric homeomorphism. 
	Since $\mathcal{E}(g, t)$ and $\mathcal{E}'\circ I'(g, t)$ are uniformly bounded away from 
	each other,  the result follows. 
	%
	%
	%
\end{proof}


The following shows that $\Uc\hat M$ can be 
used as our Gromov flow space. 

\begin{theorem}\label{thm:identify} 
	The map	
	$\mathcal{E}\circ \mathcal{F}:\Uc \hat M \ra \partial^{(2)} M \times \bR$ 
	is a quasi-isometry where each complete isometric 
	geodesic in $\Uc \hat M$ goes to 
	$(t_1,t_2)\times \bR$ for its endpoint pair $(t_1, t_2)\in \partial^{(2)}_\infty \hat M$ as an isometry. 
\end{theorem} 
\begin{proof} 
	Proposition \ref{prop:Uc} and Lemma \ref{lem:identify} imply
	the result. 
\end{proof}

\section{Partial hyperbolicity and $P$-Anosov properties} \label{sec:phyp} 


\red{The purpose of this section is to set notations and recall standard facts since some of these are not too well-known to many geometric topologists not in this specific area.}

\subsection{Some background} \label{sub:background} 
We have to repeat some of the standard materials from
Section 2.2 of \cite{GGKW17ii} and Chapter 2 of \cite{GJT}. 
Let us consider a real reductive Lie group $G$ with a maximal compact subgroup $K$. 

We assume that $\mathrm{Ad}(G) \subset \Aut(\mathfrak{g})$ for simplicity. 
For example, this holds if $G$ is connected. We can usually reduce everything here to 
the case when $G$ is connected. 
Here, $G$ is an almost product of $Z(G)_0$ and $G_s$ 
where $Z(G)_0$ is the identity component of the center $Z(G)$ of $G$ 
and $G_s = D(G)$ the derived subgroup of $G$, which is semisimple. 

A parabolic subgroup $P$ of $G$ is a subgroup $P$ of the form 
$G \cap \mathbf{P}(\bR)$ for some algebraic subgroup $\mathbf{P}$ of an algebraic group $\mathbf{G}$ 
where $\mathbf{G}(\bR)/\mathbf{P}(\bR)$ is compact
and $\mathbf{G}(\bR) = G$. 

Let $\mathfrak{z}(\mathfrak{g})$ denote the Lie algebra of $Z(G)$, 
and $\mathfrak{g}_s$ the Lie algebra of $G_s$. 
Then $\mathfrak{g} = \mathfrak{z}(\mathfrak{g}) \oplus \mathfrak{g}_s$.
Letting $\mathfrak{t}$ to denote the Lie algebra of $K$, 
we have $\mathfrak{g} = \mathfrak{t} \oplus \mathfrak{q}$ 
where $\mathfrak{q}$ is the $-1$-eigenspace of an involution $\theta$ preserving 
$\mathfrak{t}$ as the $1$-eigenspace.  
Let $\mathfrak{a} \subset \mathfrak{q}$ be the maximal abelian subalgebra, i.e., the Cartan subalgebra of 
$\mathfrak{g}$. Then $\mathfrak{a} = \mathfrak{z}(\mathfrak{g}) \cap \mathfrak{q} \oplus \mathfrak{a}_s$ 
where $\mathfrak{a}_s = \mathfrak{a} \cap \mathfrak{g}_s$, a maximal abelian subalgebra of 
$\mathfrak{q} \cap \mathfrak{g}_s$. 

We have a decomposition to $ad(\mathfrak{a})$-eigenspaces
 $\mathfrak{g}=\mathfrak{g}_{0} \oplus \bigoplus_{\alpha \in \Sigma} \mathfrak{g}_{\alpha}$.
 Here, 
 $\mathfrak{g}_{0}$ is the direct sum of $z(\mathfrak{g})$ and of the centralizer of $\mathfrak{a}$ in $\mathfrak{g}_{s} $. 
 The set $\Sigma \subset \mathfrak{a}^{*}= \mathrm{Hom}_{\mathbb{R}}(\mathfrak{a}, \mathbb{R})$ projects to a (possibly nonreduced) root system of $\mathfrak{a}_{s}^{*},$ and each $\alpha \in \Sigma$ is called a restricted root of $\mathfrak{a}$ in $\mathfrak{g}$.
 There is a subset 
$\Delta \subset \Sigma$ called a {\em simple system},  i.e. a subset such that any root is expressed 
uniquely as a linear combination of elements of $\Delta$ with coefficients all of the same sign; 
the elements of $\Delta$ are called the simple roots.
There is a subset
$\Sigma^{+} \subset \Sigma$ which is the set of positive roots, i.e. 
 roots that are nonnegative linear combinations of elements of $\Delta$, where 
 $\Sigma=\Sigma^{+} \cup\left(-\Sigma^{+}\right)$ holds additionally. 

We define $P_\theta$ to be the parabolic subgroup with Lie algebra
\[\operatorname{Lie}\left(P_{\theta}\right)=\mathfrak{g}_{0} \oplus \bigoplus_{\alpha \in \Sigma^{+}} \mathfrak{g}_{\alpha} \oplus \bigoplus_{\alpha \in \Sigma^{+} \backslash \Sigma_{\theta}^{+}} \mathfrak{g}_{-\alpha}.\]

The connected components of $\mathfrak{a} \setminus \cup_{\alpha \in \Sigma} \ker(\alpha)$ 
are called {\em Weyl chambers} of $\mathfrak{a}$. 
The component where every $\alpha \in \Sigma_+$ is positive is denoted by $\mathfrak{a}^+$.

For $G = \SL_\pm(n, \bR)$. 
The maximal abelian subalgebra 
$A_n$ of $\mathfrak{p}$ in this case
is the subspace of diagonal matrices with 
entries $a_1, \dots, a_n$ where $a_1 + \dots + a_n = 0$. 
Let $\log \lambda_i: A_n\ra \bR$ denote the projection to the $i$-th factor. 
$\Sigma$ consists of $\alpha_{ij} :=\log \lambda_i -\log \lambda_j$ for $i\ne j$, 
and $\mathfrak{sl}(n, \bR)_{\alpha_{ij}} = \{c E_{ij}| c\in \bR\}$ where $E_{ij}$ is a matrix with entry $1$ at $(i, j)$ and zero entries everywhere else. 

A positive Weyl chamber is given by
\[A_n^{+} =\{(a_1, \dots, a_n)| a_1\geq \dots \geq a_n, a_1+\cdots + a_n = 0\},\]
and we are given 
\begin{gather*} 
\Sigma^+ = \{\log \lambda_i -\log \lambda_j | i < j\} \hbox{ and }\\ 
\Delta = \{\alpha_i:= \log \lambda_i -\log \lambda_{i+1} | i=1, \dots, n-1\}.
\end{gather*}  

We recall Example 2.14 of \cite{GGKW17ii}: 
Let $G$ be $\mathrm{GL}(n, \bR)$ seen as a real Lie group. Its derived group is $G_{s}=D(G)=\mathrm{SL}(n, \bR) .$ 
We can take 
$\mathfrak{a} \subset \mathfrak{g l}(n, \bR)$ to 
be the set of real diagonal matrices of size $n \times n$. 
As above, for $1 \leq i \leq n,$ let $\log \lambda_{i} \in \mathfrak{a}^{*}$ 
be the evaluation of the $i$-th diagonal entry. 
Then 
\[\mathfrak{a}=\mathfrak{z}(\mathfrak{g}) \cap \mathfrak{a} \oplus \mathfrak{a}_{s} 
\hbox{ where }
\mathfrak{z}(\mathfrak{g}) \cap \mathfrak{a}=\bigcap_{1 \leq i, j \leq n}  \operatorname{Ker}\left(\log \lambda_{i}-\log \lambda_{j}\right)\] 
is the set of real scalar matrices, and $\mathfrak{a}_{s}=
\operatorname{Ker}\left(\log \lambda_{1}+\cdots+\log \lambda_{n}\right)$ 
is the set of traceless real diagonal matrices. The set of restricted roots of $\mathfrak{a}$ in $G$ is
$$
\Sigma=\left\{\log \lambda_{i}-\log \lambda_{j} \mid 1 \leq i \neq j \leq n\right\}.
$$
We can take $\Delta=\left\{\log \lambda_{i}-\log \lambda_{i+1} \mid 1 \leq i \leq n-1\right\}$ so that
$$
\Sigma^{+}=\left\{\log \lambda_{i}-\log \lambda_{j} \mid 1 \leq i<j \leq n\right\},
$$
and $\mathfrak{a}^{+}$ is the set of the elements of $\mathfrak{a}$ whose entries are in increasing order. 
Let $\clo({\mathfrak{a}}^+)$ denote its closure. 
For 
\begin{multline} 
\theta=\left\{\log \lambda_{i_{1}}-\log \lambda_{i_{1}+1}, \ldots, \log \lambda_{i_{m}}-\log \lambda_{i_{m}+1}\right\} \\ 
\hbox{ with } 1 \leq i_{1}<\cdots<i_{m} \leq n-1,
\end{multline} 
the parabolic
subgroup $P_{\theta}$ (resp. $\left.P_{\theta}^{-}\right)$ is the group of block upper (
resp. lower) triangular matrices in $\mathrm{GL}(n,\bR)$ with square diagonal blocks 
of sizes $i_{1}, i_{2}-i_{1}, \ldots, i_{m}-i_{m-1}, n-i_{m}$. 
In particular, $P_{\Delta}$ is the group of upper triangular matrices in $\mathrm{GL}(n,(\bR)$.

From now on, we will fix $\mathfrak{a}_n \subset \GL(n, \bR)$ 
to be the set of real diagonal matrices of size $n \times n$. 
We will also regard $\theta$ as a subset of $\{1, \dots, n-1\}$ where 
each $i$ corresponds to $\log \lambda_i - \log \lambda_{i+1}$. 
Also, let $\mathfrak{a}_n^+ $ denote the open positive Weyl chamber.

\subsection{$P$-Anosov representations}  \label{sub:Panosov} 
Let $(P^{+}, P^{-})$ be a pair of opposite parabolic subgroups of $\GL(n, \bR)$, and 
let $\mathcal{F}^{{\pm}}$ denote the flag spaces. 
Let $\chi$ denote the unique open $\GL(n, \bR)$-orbit of the product 
$\mathcal{F}^{+}\times \mathcal{F}^{-}$. 
The product subspace $\chi$ has two $\GL(n, \bR)$-invariant 
distributions $E^{\pm} := T_{x_{\pm}} \mathcal{F}^{\pm}$ for $(x^{+}, x^{-}) \in \chi$. 
Let $\phi_t$ be given as in Theorem \ref{thm:Mineyev} as the $\bR$-action on 
$F \Gamma_M$  whose orbits are quasi-geodesics. 
We denote by $(\tau_+, \tau_-): F\Gamma_M \ra F\Gamma_M/\bR \cong \delta_\infty \hat \Gamma_M^{(2)}$ the maps associating to a point the end points of its $\bR$-orbit.


\begin{definition}[Definition 2.10 of \cite{GW12}]\label{defn:Anosov} 
	A representation $\rho: \Gamma_M \ra \GL(n, \bR)$ is {\em $(P^{+}, P^{-})$-Anosov} if 
	there exist continuous $\rho$-equivariant maps $\xi^+: \partial_\infty \hat \Gamma_M \ra {\mathcal{F}}^+$, 
	$\xi^-:\partial_\infty \hat \Gamma_M \ra {\mathcal{F}}^-$ such that:  
	\begin{enumerate} 
		\item[(i)] for all $(x, y)\in \partial_\infty^{(2)} \hat\Gamma_M$, the pair 
		$(\xi^+(x), \xi^-(y))$ is transverse, and  
		\item[(ii)] for one (and hence any) continuous and equivariant family of norms 
		$(\llrrV{\cdot}_{m})_{m\in F\Gamma_M} $ on 
		\[(T_{\xi^+(\tau^+(m))} {\mathcal{F}}^+)_{m\in F\Gamma_M}
		\Big(\hbox{resp. } 
		(T_{\xi^-(\tau^-(m))} {\mathcal{F}}^-)_{m\in F\Gamma_M}\Big):\]
		\[ \llrrV{e}_{\phi_{-t}m} \leq A e^{-at} \llrrV{e}_{m}  \Big(\hbox{resp.} 
		\llrrV{e}_{\phi_{t}m} \leq A e^{-at} \llrrV{e}_{m} \Big).
		   \]
	\end{enumerate}
We call 	$\xi^\pm:\partial_\infty F\Gamma_M \ra {\mathcal{F}}^\pm$ the {\em Anosov maps} associated with 
$\rho: \Gamma_M \ra G$. 
	(See also Section 2.5 of  \cite{GGKW17ii} for more details.)
\end{definition}

It is sufficient to consider only the case where $P^{+}$ is conjugate to $P^{-}$
by Lemma 3.18 of \cite{GW12}. (See also Definition 5.62 of \cite{KLP17}). 
In this case, $\rho$ is called {\em $P$-Anosov} for $P = P^+$. 
Note that this means that 
\begin{equation}  \label{eqn:theta*} 
\theta = \theta^*
\end{equation} 
where $\theta^*$ is the image $\theta$ under the map $i \mapsto n-i$, $i=1, \dots, n-1$.
Hence, this will always be true for the set $\theta$ occurring for $P_\theta$-Anosov 
representations below.
(See the paragraph before Fact 2.34 of \cite{GGKW17ii}.)

\subsection{$P$-Anosov representations and dominations}



Let $\rho: \Gamma_M \ra \GL(n, \bR)$ denote a representation. 
We order singular values 
\[a_1(g) \geq a_2(g) \geq \dots \geq a_n(g)
\hbox{ of } \rho(g) \hbox{ for } g \in \rho(\Gamma_M).\] 
%
%
%
%
%
%
%
%
%
%
%
%
%
Let $w(g)$ denote the word length of $g$. 
Suppose that there exists an integer $k$,
$1 \leq k \leq n-k+1 \leq n$, 
so that the following hold for a constant $A, C > 0$: 
\begin{itemize}  
	\item $a_k(\rho(g))/a_{k+1}(g) \geq C\exp(A w(g))$, and
	\item $a_{n-k}(\rho(g))/a_{n-k+1}(g) \geq C\exp(A w(g))$.
\end{itemize} 
In this case, we say that $\rho$ is {\em $k$-dominated}
for $k \leq n/2$ 
(see Bochi-Potrie-Sambarino \cite{BPS}).





%
%


We will use $\llrrV{\cdot}$ to indicate the Euclidean norm in the maximal flat 
in a symmetric space $X$. (See Example 2.12 of \cite{KL18}.)
We will use the following notation: 
\[\vec{a}(g) :=(a_1(g), \dots, a_n(g)), \quad
\log \vec{\lambda}(g) := (\log \lambda_1(g), \dots, \log \lambda_n(g)) \in \clo(\mathfrak{a}^+_n)\] 
for the singular values $a_i(g)$ and 
the modulus $\lambda_i(g)$ of the eigenvalues of $g$
for $g \in \GL(n, \bR)$.
%

Given a set $A$, 
we say that two functions $f, g:A \ra \bR$ are {\em compatible} if 
there exists a uniform constant $C>1$ so that 
$C^{-1} (x) < f(x) < C g(x)$ for all $x\in A$. 

\begin{lemma}\label{lem:Equiv} 
	Let a finitely generated discrete affine transformation group $G$ have  linear holonomy representation $\rho:G \ra \GL(n, \bR)$.

	Then the following are equivalent\/{\em :} 
	\begin{itemize} 
	\item  $\rho$ is $P$-Anosov in the Gu\'eritaud-Guichard-Kassel-Wienhard sense \cite{GGKW17ii}
	for some parabolic group $P = P_\theta$ for an index set $\theta$ containing  $k$
	and $n-k$. 
	\item  $\rho$ is a $k$-dominated representation 
	for $1 \leq k \leq n/2$. 
	\end{itemize} 
\end{lemma} 
\begin{proof} 
	This follows by Theorem 1.3(3) of \cite{GGKW17ii}
	since the set $\theta$ contains $k$ and $n-k$.  
\end{proof}

\section{The proof of Theorem  \ref{thm:main2}} \label{sec:PAnosov}

\subsection{$k$-convexity} \label{sub:kconvexity} 
\red{We will relate the partial hyperbolic property in the singular value sense to that in the bundle sense
	following \cite{BPS}.  We need this modification since we have to have actual expanding and shrinking under the flow.}

We continue to assume that $\Gamma_M$ is hyperbolic. 
Let $\rho:\Gamma_M \ra \GL(n, \bR)$ be a representation. 
We say that the representation $\rho$ is {\em  $k$-convex} if there exist continuous maps 
$\zeta: \partial_\infty \hat \Gamma_M \ra \mathcal{G}_k(\bR^n)$ and $\theta:\partial_\infty \hat \Gamma_M\ra \mathcal{G}_{n-k}(\bR^n)$ such that:
\begin{description} 
	\item[(transversality)] for every $x,y\in\partial_\infty \hat \Gamma_M,x\ne y$, we have $\zeta(x)\oplus\theta(y)=\bR^n$, and
	\item[(equivariance)] for every $\gamma \in \Gamma_M$, 
	we have \[\zeta(\gamma x) = \rho(\gamma)\zeta(x), \theta(\gamma x) = \rho(\gamma)\theta(x), x \in \partial_\infty \hat \Gamma_M.\]
\end{description}


Using the representation $\rho$, it is possible to construct a linear 
flow $\psi_t$ over the geodesic flow $\phi_t$ of 
the Gromov flow space $\partial^{(2)}_\infty \hat M \times \bR$ as follows. 
Consider the lifted geodesic flow $\hat \phi_t$ on 
$\partial^{(2)}_\infty \hat M \times \bR$, and
define a linear flow on $\hat E := 
(\partial^{(2)}_\infty \hat M \times \bR)\times\bR^n$ by:
$\hat\psi_t(x,\vec{v})=(\phi_t(x), \vec{v})$ for $x\in\partial^{(2)}_\infty \hat M \times \bR$.
Now consider the action of $\Gamma_M$ on $\hat E$ given by:
\[ \gamma \cdot \left(x, \vec{v}\right):= 
\left(\gamma(x), \rho(\gamma)(\vec{v})\right). \] 
It follows that $\hat \psi_t$ induces in a flow on $E:=\bR^n_{\rho}= \hat E/\Gamma_M$. 

When the representation $\rho$ is $k$-convex, by the equivariance, there exists a $\phi_t$-invariant splitting of the form $E_\rho = Z \oplus \Theta$; 
it is obtained by taking the quotient of the bundles
\[\hat Z(x) = \zeta(x_+) \hbox{ and }
\hat \Theta(x) = \theta(x_-)
\hbox{ for } x \in \partial^{(2)}_\infty \hat M \times \bR\]
where $x_+$ is the forward endpoint of 
the complete isometric geodesic through $x$ and 
$x_-$ is the backward endpoint of the complete isometric geodesic through $x$. 

We say that a $k$-convex representation is {\em $k$-Anosov
in the bundle sense} if the splitting 
$E_\rho=Z\oplus\Theta$ is a dominated splitting for the linear bundle automorphism $\psi_t$, with $Z$ dominating $\Theta$ in the terminology of \cite{BPS}. 
This is equivalent to the fact that the bundle $\Hom(Z,\Theta)$ is uniformly contracted
by the flow induced by $\psi_t$.

Suppose that we further require for a $k$-dominated representation $\rho$
 for $k \leq n/2$\/: 
\begin{itemize}  
	\item $a_p(\rho(g)) \geq C^{-1} \exp(A w(g))$ for $p \leq k$, and 
	\item $a_r(\rho(g)) \leq C \exp(-A w(g))$ for $r \geq n-k+1$
	for some constants $C> 1, A> 0$. 
\end{itemize} 
Then we say that $\rho$ is {\em partially hyperbolic
	in the singular value sense with index $k$}. 
 
We obtain by Lemma \ref{lem:Equiv}:
\begin{lemma}\label{lem:phypAnos} 
	If $\rho$ is partially hyperbolic for an index $k$, $k \leq n/2$, in the singular value sense, 
	then $\rho$ is $P_\theta$-Anosov for the index set $\theta$ containing $k$ and $n-k$. 
	\qed
\end{lemma} 

Recall that we can identify $\partial_\infty \hat \Gamma_M$ for the Cayley graph $\hat \Gamma_M$ of $\Gamma_M$ with 
$\partial_\infty \hat M$ by Theorem 11.108 of \cite{DK2018}. 
Hence, we can identify $\partial_\infty^{(2)} \hat \Gamma_M$ with $\partial^{(2)}_\infty \hat M$. 
The Gromov flow space $\partial_\infty^{(2)} \hat \Gamma_M \times \bR$ is identified with 
$\partial^{(2)}_\infty \hat M \times \bR$. 
By the direct following of Proposition 4.9 of 
Bochi-Potrie-Sambarino \cite{BPS}, we obtain: 
\begin{theorem} \label{thm:bundle} 
	Let $\rho$ be partially hyperbolic with an index $k$, $k < n/2$, 
	in the singular value sense. 
	Then 
	$\rho$ is partially hyperbolic in the bundle sense with index $k$.
\end{theorem} 
\begin{proof} 
	By Lemma \ref{lem:phypAnos} and Proposition \ref{lem:Equiv}, 
	$\rho$ is $k$-dominated.

	Bochi, Potrie, and Sambarino \cite{BPS} 
	construct subbundles $Z$ and $\Theta$ of $E_\rho$ 
	where $Z$ dominates $\Theta$
	and $\dim Z = k, \dim \Theta = n-k$. 

	Now, $\rho$ is also $n-k$-dominated. 
	By Lemma \ref{lem:phypAnos}, 
	we obtain a new splitting of the bundle $E_\rho = Z'\oplus \Theta'$ 
	over $\partial_\infty^{(2)}\hat M \times \bR$
	where $\dim Z' = n-k$ and $\dim \Theta'$ is $k$. 
	Furthermore, $Z'$ dominates $\Theta'$. 
	Then $Z$ is a subbundle of $Z'$, and 
	 $\Theta'$ is a subbundle of $\Theta$ by 
	Proposition 2.1 of \cite{BPS}. 
	%
	We now form bundles $Z, Z'\cap \Theta, \Theta'$ over $\partial_\infty^{(2)}\hat M \times \bR$. 
	The dominance property \eqref{eqn:dominance} of Definition \ref{defn:phyp}  follows from those of $Z, \Theta$ and $Z', \Theta'$. 
		
	There are exponential expansions 
	of $a_i(g)$ for $1\leq i \leq k$ by the partially 
	hyperbolic condition of the premise. 
	Now (iii)(a) of Definition \ref{defn:phyp} follows: 
	Define $U_p(A)$ as the sum of eigenspaces of $\sqrt{AA^*}$ corresponding to $p$ largest eigenvalues. 
	Define $S_{n-p}(A):= U_{n-p}(A^{-1})$. 
	Since \red{$\hat M/\Gamma_M$} is compact, Theorem 2.2 of \cite{BPS} shows 
	\[U_k(\psi^t_{\phi^{-1}_t(x)})\ra \hat Z(x) \hbox{ and } 
	S_{n-k}(\psi^t_x) \ra \hat \Phi(x) \hbox{ uniformly for $x \in \hat \Gamma_M$  as } t \ra \infty.\] 

	Since by the paragraph above Theorem 2.2 of \cite{BPS}, we obtain
	\[\psi^t(S_{n-k}(\psi^t_{\phi^{-1}_t(x)})^{\perp}) = U_k(\psi^t_{\phi^{-1}_t(x)}),\]
	where the unit vectors of $S_{n-k}(\psi^t_{\phi^{-1}_t(x)})^{\perp}$ goes to 
	vectors of norm at least the $k$-th singular value for $\psi_t$. 
		The angle between $\hat Z(x)$ and $\hat \Phi(x)$ is uniformly bounded below by 
a positive constant 
		by the compactness of \red{$\hat M/\Gamma_M$}. 
		Hence, unit vectors in $\hat Z({\phi^{-1}_t(x)})$ have 
		components in $S_{n-k}(\psi^t_{\phi^{-1}_t(x)})^\perp$ whose lengths are uniformly bounded below by a positive constant 
		for sufficiently large $t$. 

	As $t \ra \infty$, $\phi_t$ passes the images of the fundamental domain under $\Gamma_M$. 
	Let $g_i$ denote the sequence of elements of $\Gamma_M$ arising in this way. 
		Since the $a_k(g_i)$ grows exponentially, the expansion properties of $Z$ follow
\red{by the above paragraph}. 
	(iii) of  Definition \ref{defn:phyp}  follows up to reversing the flow. 
	
			By the quasi-isometry of $\Uc \hat M$ with $\partial^{(2)}\hat M \times \bR$ in Theorem \ref{thm:identify},
	we construct our pulled-back partially hyperbolic bundles over $\Uc \hat M$.
%
%
%
	%
\end{proof}

\subsection{Promoting $P$-Anosov representations to  
partially hyperbolic ones}
\label{sub:converse} 

\red{This subsection contains the main results of Part 1. The complication is that the maximal abelian subgroup of the Zariski closure of the linear part of the holonomy group sometimes embeds not linearly but piecewise linearly. Also, we have to understand the case when the Zariski closure is not reductive.}

	\begin{theorem}[Hirsch-Kostant-Sullivan \cite{KS75}]\label{thm:HKS}
	Let $N$ be a complete affine manifold.
	Let $\rho:\pi_1(N) \ra \GL(n, \bR)$ be the linear part of the affine holonomy representation $\rho'$.
	Then $g$ has an eigenvalue equal to $1$ for each $g$ in the Zariski closure $Z(\rho(\pi_1(N)))$ of $\rho(\pi_1(N))$. 
\end{theorem}
	\begin{proof} 
		We define function $f: Z(\rho(\pi_1(N))) \ra \bR$ 
		by $f(x) = \det(x-\Idd)$ for each element $x$. 
		Suppose that $f(\rho(g)) \ne 0$, $g\in \Gamma_M$. Then $\rho(g)(x) + b = x$ has 
		a solution for each $b$ by Cramer's rule. 
		Hence $\rho'(g)$ has a fixed point, which is a contradiction.
		We obtain that $f$ is zero on the Zariski closure, 
		implying that each element has at least one eigenvalue equal to $1$.  
		\end{proof} 

For a reductive Lie group $Z$,  
let $A_Z$ denote the maximal $\bR$-split torus of $Z$ with corresponding 
Lie algebra $\mathfrak{a}_Z$, and 
 let $\mathfrak{a}_Z^+$ denote the Weyl chamber of $\mathfrak{a}_Z$.
Also, we denote 
the Jordan (or Lyapunov) projection by $\log \vec{\lambda}_Z: Z \ra \clo(\mathfrak{a}^+_Z)$. 
Let $\Gamma$ be Zariski dense in $Z$. (See Section 2.4 of \cite{GGKW17ii}.)
The {\em Benoist cone}  $\mathcal{BC}_Z(\Gamma)$ of a linear group 
$\Gamma \subset Z$ is the closure of the set of 
all positive linear combinations of 
$\log \vec{\lambda}_Z(g), g\in \Gamma$ in $\clo(\mathfrak{a}^+_Z)$.

The following was worked out with the help of Danciger and Stecker: 
\begin{proposition}\label{prop:PAnosovStrict}  
Assume the following\/{\em :} 
\begin {itemize}
	\item  $G$ is a finitely generated discrete group. 
	\item $\rho:G \ra \GL(n, \bR)$ is a
	$P_\theta$-Anosov representation 
	for some index set $\theta =\theta^*$ 
	containing $k$ and $n-k$ for 
	$1\leq k \leq n/2$.
	\item $\rho(g)$ for each $g\in G$ has an eigenvalue equal to $1$. 
	\item The Zariski closure $Z$ of the image of $\rho$ is 
	a connected reductive Lie group. 
\end{itemize} 
	Then $\rho$ is partially hyperbolic with index $k$ in the singular value sense, 
	and $k < n/2$.   
	\end{proposition}
\begin{proof} 
	(I) We begin with some preliminary:  
%
	 	The embedding $Z \hookrightarrow \GL(n, \bR)$
	 induces a map $\iota_Z:\clo(\mathfrak{a}_Z^+) \ra \clo(\mathfrak{a}^+_n)$ for the Weyl chamber
	 $\mathfrak{a}^+_n$ of $\GL(n, \bR)$. 
	 Let $\log\lambda_1, \dots, \log\lambda_n$ denote the 
	 coordinates of the Weyl chamber $\mathfrak{a}^+_n$, 
which actually are standard coordinates.
	 Of course, $\iota_Z \circ \log \vec{\lambda}_Z| 
	 \rho(G) = \log\vec{\lambda}|\rho(G)$. \

	 	Let $B:=\mathcal{BC}_Z(\rho(G))$ denote the 
	 Benoist cone of $\rho(G)$ in $Z$. 
	  The image $\iota_Z(B)$ is not convex in general. 
	 Now, $B$ is a convex cone with nonempty interior in $\clo(\mathfrak{a}_Z^+)$. 	 

	 By Theorem \ref{thm:HKS}, 
	 for each $b \in B$, there is an index $i$ where 
	 the log of the $i$-th coordinate of $\iota_Z(b)$ equals $0$. 
We denote by 
\[\mathcal{I}(b)=\{ i | \log \lambda_i(\iota_Z(b)) = 0\},\] 
which is a consecutive set. 
	 Let $S([1, ..., n])$ denote the collection of 
	 subsets of consecutive elements.
	 We define an integral interval function
	 \[\mathcal{I}: B -\{0\} \ra S([1, \dots, n])\]
	 given by sending $b \in B$ to $\mathcal{I}(b)$. 
 Since the sequences of
ordered norms of eigenvalues of convergent sequence of linear maps converge,
we have 
	 \begin{equation} \label{eqn:limI} 
	 \lim_{i\ra \infty} \mathcal{I}(x_i) \subset I(x) \hbox{ whenever } 
	 x_i \ra x, x_i, x\in B.
	 \end{equation}

 Let $S$ denote a finite set of generators of $\rho(G)$ and their inverses. 
 By Theorems 1.3 and 1.4 of \cite{BS2021}, the sequence
 $\frac{1}{m} \log\vec{\lambda}_Z(S^m)\subset J(S)$ geometrically converges to 
 a compact convex set $J(S) \subset\clo(\mathfrak{a}_Z^+)$  as $m \ra \infty$ 
 (see Section 3.1 of \cite{BS2021}).
 Moreover, the cone spanned by $0$ and $J(S)$ is the Benoist cone $B$. 
 This implies that the set of directions of 
 $\log \vec{\lambda}_Z(\rho(G))$
 is dense in $B$. (See also \cite{Benoist98}). 
 Let $S_n$ denote the unit sphere in the Lie algebra $\mathfrak{a}_n$ of $\GL(n, \bR)$
with respect to the standard coordinates.
 Since 
 \[\log \lambda_k(g) -\log \lambda_{k+1}(g) \geq C \llrrV{\log \vec{\lambda}(g)}, g \in \rho(G),
 \hbox{ for a constant } C> 0 \]
 by Theorem 4.2(3) of \cite{GGKW17ii},  we obtain
 \begin{equation} \label{eqn:lklkpC} 
 	\log \lambda_k(b) -\log \lambda_{k+1}(b) \geq C 
 \end{equation} 
for every $b  \in B -\{O\}$ going to an element of $\iota_Z(B) \cap S_n$ 
for a constant $C> 0$.
	 
	    (II) Our first major step is to prove
	 $\mathcal{I}(B-\{0\}) \subset [k+1, n-k]$: 
	  Let $\iota_n$ denote the involution of $\mathfrak{a}_n$. 
First, suppose that the rank of $Z$ is $\geq 2$, and hence 
$\dim \mathfrak{a}_Z^+ \geq 2$.

A {\em diagonal subspace }
\[\Delta_{i, j} \subset \mathfrak{a}_n \hbox{ for } i\ne j, 
i, j=1, \dots, n,\] 
is a subspace of the maximal abelian algebra
$\mathfrak{a}_n$ given by 
$\log\lambda_i - \log\lambda_j=0$ for some indices $i$ and $j$. 
	The inverse images under $\iota_Z$ in $\clo(\mathfrak{a}_Z^+)$ of the 
	diagonal subspaces of $\mathfrak{a}$ may meet $B$. 
	We let $\mathcal{I}_B$ the maximal collection of indices 
	$(i, j), i < j$, where
	$\iota_Z^{-1}(\Delta_{i, j})$ contains $B$, which may be empty.  
	Let $\mathcal{I}'_B$ the collection of indices $(k, l)$, $k< l$, of 
	$\Delta_{k, l}$ not in $\mathcal{I}_B$. 
	We call the closures of the components of 
	\[B - \bigcup_{(k, l)\in \mathcal{I}'_B} \iota_Z^{-1}(\Delta_{k, l}) \] 
	the {\em generic flat domains of $B$}.  
	There are finitely many of these 
	to be denoted $B_1, \dots, B_m$ which are convex 
	domains in $B$.

We make notes of two ``discrete continuity" properties due to Danciger and Stecker: 
\begin{itemize} 
\item	In the interior of each $B_i$, $\mathcal{I}$ is constant since the 
$\mathcal{I}$-value is never empty and the change  in $\mathcal{I}$-values 
will happen only in the inverse images of some of $\Delta_{k,l}$. 
	
	
\item	For adjacent regions, their images under $\mathcal{I}$ differ by
	adding or removing some top or bottom consecutive sets. 
\end{itemize} 
Heuristically speaking: the worms leave ``traces".
	
	We now aim to prove that the following never occurs --(*): 
	\begin{quotation} 
	There exists a pair of elements $b_1, b_2 \in B$
	with $b_1 \in B_p^o$ and $b_2 \in B_q^o$ for some $p, q$
	where $\mathcal{I}(b_1)$ has an element $\leq k$ and 
	$\mathcal{I}(b_2)$ has an element $> k$.
	\end{quotation} 
(The idea is that $k, k+1$ both cannot be in a zero set since they differ by a constant due to Anosov properties. )

We use the discrete version of continuity: 
	Suppose that this (*) is true. 
	Since $B$ is convex with $\dim B \geq 2$, 
	there is a chain of adjacent polyhedral images 
	$B_{l_1}, \dots, B_{l_m}$ where $b_1 \in B_{l_1}^o$
	and $b_2 \in B_{l_m}^o$
	where $B_{l_j}\cap B_{l_{j+1}}$ contains a nonzero point.  
	For adjacent $B_{l_p}$ and $B_{l_{p+1}}$, 
	we have by \eqref{eqn:limI} 
\begin{equation} \label{eqn:intI} 
 \mathcal{I}(B_{l_p}^o)\cup \mathcal{I}(B_{l_{p+1}}^o) \subset \mathcal{I}(x)  \hbox{ for }
x\in \clo(B_{l_p})\cap \clo(B_{l_{p+1}}) -\{O\}. 
\end{equation} 
	Considering $\mathcal{I}(B_{l_j}^o)$, $j=1, \dots, m$, and 
	the $\mathcal{I}$-values of the intersections of adjacent generic flat domains, we have 
$\mathcal{I}(b') \ni k, k+1$ for some nonzero element $b'$ of $B$ by 
the connectedness of the $\mathcal{I}$-values as we change
generic flat domains to adjacent ones according to 
\eqref{eqn:intI}.
Hence, we obtain
\begin{equation}\label{eqn:kk+1}  
\log \lambda_k(b') = 0 
\hbox{ and }\log \lambda_{k+1}(b')=0.
\end{equation}

	Since we can assume $\iota_Z(b') \in S_n$ for $b'$ in \eqref{eqn:kk+1}
	by normalizing, we obtain a contradiction to \eqref{eqn:lklkpC}.

	Hence, we proved (*), and  
only	one of the following holds:
	\begin{itemize} 
	\item for all $b\in B$ in the interiors of generic flat domains,  every element of $\mathcal{I}(b)$ is $\leq k$ 
	or 
	\item for all $b \in B$ in the interiors of generic flat domains, every element of 
	$\mathcal{I}(b)$ is $> k$.
	\end{itemize} 

Note that the opposition involution $\iota_n$ sends $\iota_Z(B)$ to itself
since the inverse map $g\ra g^{-1}$ preserves $Z$ 
and the set of eigenvalues of $g$ is sent to the set of 
the eigenvalues of $g^{-1}$.
	In the first case, by the opposition involution $\iota_n$ preserving $G$, 
	$\mathcal{I}(\iota_n(b))$ for $b \in B$ contains an element $\geq n-k+1$.
	This again contradicts the above. 
	Hence, we conclude $\mathcal{I}(b) \subset [k+1, \dots, n]$ for all $b \in B$
	in the interiors of generic flat domains. 
Acting by the opposition involution and by similar arguments, we obtain 
$\mathcal{I}(b)\subset [k+1, \dots, n-k]$ for all $b \in B$ in the interiors of generic flat domains. 

Furthermore, we can show that $\mathcal{I}(b) \subset [k+1, n-k]$ for all $b \in B -\{0\}$ since otherwise, 
we have $\mathcal{I}(b') \ni k, k+1$ for some $b'\ne 0$ by a similar reason to the above applied to lower-dimensional strata. 
This is a contradiction as before. 
Also, the argument shows $[k+1, n-k] \ne \emp$ and $k < n/2$.

	 Suppose that $Z$ has rank $1$. 
Then the maximal abelian group $A_Z$ is $1$-dimensional. 
For each $g \in Z$, there is an 
index $i(g)$ for which $\log\lambda_{i(g)}(g) =0$. 
Thus, $\mathfrak{a}_Z$ is in the null space of $\log \lambda_i$ for some index $i$. 
Since $g \mapsto g^{-1}$ preserves $\mathfrak{a}_Z$, 
we have $\log \lambda_{n-i+1} = 0$ on $\mathfrak{a}_Z$. 
Since $\clo(\mathfrak{a}_Z^+) = \mathfrak{a}_Z \cap \clo(\mathfrak{a}^+)$, we also have 
$\log \lambda_{n-i+1}=0$ on $\mathfrak{a}_Z^+$. 
Hence, 
\[\mathcal{I}(B-\{0\}) = [i, \dots, n-i+1] \hbox{ for some index } i, 0< i \leq n/2.\] 
Now \eqref{eqn:lklkpC} and $k < i$ imply $\mathcal{I}(B-\{0\}) \subset [k+1, n-k]$. 

	(III) To complete the proof, we show the 
	partial hyperbolicity property: 
We let $\log \vec{a}_Z|Z: Z \ra \mathfrak{a}_Z^+$ denote the Cartan projection of $Z$. 
Recall 
\[\iota_Z\circ \log \vec{a}_Z| \rho(G) 
= \log \vec{a}|\rho(G).\] 
Since $\rho(G)$ has a connected reductive Zariski closure, 
the sequence 
\[\left\{\frac{\log \vec{a}(\rho(g_l)))}{w(g_l)}\right\}\]
have limit points only in $\iota_Z(B)$
by Theorem 1.3 of \cite{BS2021}. 
By  the conclusion of step (II), 
for any sequence $\{g_l\}$, 
\[\left\{\frac{\min\{|\log a_i(\rho(g_l))|\}_{i=k+1, \dots n-k}}{w(g_l)}\right\} \ra 0 ,\]
uniformly in terms of $w(g_l)$. 
Hence, this  implies that for any small constant $C> 0$, there is $N> 0$ such that 
\[a_i(\rho(g)) \geq \exp(-C w(g))\] for at least one $i \in [k+1, \dots, n-k]$ 
provided $w(g) > N$. 
This means that 
\[a_i(\rho(g)) \geq A \exp(-C w(g))\] for all $g$ for a constant $A > 0$ depending on $C$. 
By the $P_\theta$-Anosov property of $\rho$ and Lemma \ref{lem:Equiv}, we have 
\[\frac{a_k(\rho(g))}{a_{k+1}(\rho(g))}\geq A'\exp(C'w(g)), g \in \Gamma\] 
for some constants $A', C'> 0$. 
Since 
\[a_{k+1}(\rho(g)) \geq a_i(\rho(g)), i=k+1, \cdots, n-k, 
\hbox{ for every } g \in G,\] 
it follows that 
\[a_k(\rho(g)) \geq A'' \exp(C'' w(g))\] for some constants $C'', A''> 0$. 
Taking inverses also, we showed that $\rho'$ is partially hyperbolic 
in the singular value sense with index $k$. 

Finally, since $[k+1, n-k]$ is not empty, $k < n/2$.  
	\end{proof}

For a general representation $\phi: \Gamma \ra G$, we define the {\em semisimplification} 
$\phi^{ss}: \Gamma \ra G$ by post-composing $\phi$ with the projection to 
the Levi factor of the Zariski closure of $\phi(\Gamma)$.  
(See Section 2.5.4 of \cite{GGKW17ii} for details.)
Also, the Zariski closure of a finite index group is reductive if so was the Zariski closure of the original discrete group.
(See Remark 2.38 of \cite{GGKW17ii}.) 
	
	\begin{proof}[Proof of Theorem  \ref{thm:main2}] 
		Suppose that $\rho$ is $P$-Anosov in the sense of \cite{GGKW17ii}.
		Hence, $\Gamma:= \rho(\pi_1(N))$ is word-hyperbolic by \cite{GGKW17ii}. 
	
		First, suppose that the Zariski closure of the image of the linear holonomy group is 
	reductive. 
	By Lemma \ref{lem:Equiv}, $\rho$ is a $k$-dominated
	representation of index $k$. 
We take a finite index normal subgroup $\Gamma'$ of $\Gamma$. 
	Assume that the Zariski closure $Z$ of $\rho(\Gamma')$ is a connected 
	reductive Lie group. 	 
	By Theorem \ref{thm:HKS} and 
	Proposition \ref{prop:PAnosovStrict}, $\rho|\Gamma'$ is 
	partially hyperbolic in the singular value sense with index $k$, $k < n/2$.
   So is $\rho$ since $\Gamma$ is a finite-index extension of $\Gamma'$. 
	By Theorem \ref{thm:bundle}, $\rho$ is a partially hyperbolic 
	representation in the bundle sense with index $k$.

	
	
	Now, we drop the reductive condition of the Zariski closure. 
Mainly, we have to prove the expanding and contracting properties in the subbundles. 
	The semisimplification $\rho^{ss}$ is $P$-Anosov by Proposition 1.8 of \cite{GGKW17ii}. 
	Since the projection to the Levi factor does not change eigenvalues by the proof of Lemma 2.40 of \cite{GGKW17ii}, 
$\rho^{ss}(g)$	for each $g \in \Gamma$ has an eigenvalue equal to $1$ by	Theorem \ref{thm:HKS}. 
	By Proposition \ref{prop:PAnosovStrict} and Theorem \ref{thm:bundle}, 
	$\rho^{ss}$ is a partially hyperbolic representation in the bundle sense with index $k$, $k < n/2$. 
	By Proposition 1.3 of \cite{GGKW17ii}, $\rho$ is $P$-Anosov with respect to 
	a continuous Anosov map $\zeta: \partial_\infty \Gamma \ra {\mathcal{F}}$ for the flag variety 
	$\mathcal{F} := \GL(n, \bR)/P$.

We can conjugate $\rho$ as close to $\rho^{ss}$ as one wishes in $\Hom(\Gamma, \GL(n, \bR))$
since the orbit of $\rho^{ss}$ is the closed orbit in the closure of the orbit of $\rho$ under the conjugation 
action. (See Section 2.5.4 of \cite{GGKW17ii}. )
Let $\zeta_i: \partial_\infty \Gamma \ra {\mathcal{F}}$ denote the Anosov map for $\rho_i$ 
conjugate to $\rho$ converging to $\rho^{ss}$. 
Theorem 5.13 of \cite{GW12} generalizes to $\GL(n, \bR)$ 
since we can multiply a representation 
by a homomorphism $\Gamma \ra \bR^+$ not changing the P-Anosov property
to $\SL_\pm(n, \bR)$-representations. 
(See Section 2.5.3 of \cite{GGKW17ii}). 
Hence, Theorem 5.13 implies that
$\zeta_i$ converges to $\zeta^{ss}: \partial_\infty \Gamma \ra {\mathcal{F}}$
the map for $\rho^{ss}$ in the $C^0$-sense. 

Let $\bV^{ss+}, \bV^{ss0}$, and $\bV^{ss-}$ denote the bundles of the partially hyperbolic decomposition of $\bR^n_{\rho^{ss}}$ 
as in Definition \ref{defn:phyp}. 
Let $\bV^{+}_i, \bV^{0}_i$, and $\bV^{-}_i$ denote the bundles of the decomposition of $\bR^n_{\rho_i}$ 
of respective dimensions $k$, $n-2k$, and $k$ obtainable from $\zeta_i$
since $\rho_i$ is $k$-dominated.  
The  single product bundle $\Uc \hat M \times \bR^n$ covers the direct sums of these bundles.
\begin{itemize} 
\item Let $\hat \bV^{ss+}, \hat \bV^{ss0}$, and $\hat \bV^{ss-}$ denote the cover of 
$\bV^{ss+}, \bV^{ss0}$, and $\bV^{ss-}$. 
\item Let $\hat \bV^{+}_i, \hat \bV^{0}_i$, and $\hat \bV^{-}_i$ denote the cover of 
$\bV^{+}_i, \bV^{0}_i$, and $\bV^{-}_i$. 
\end{itemize} 
Then the above convergence $\zeta_i \ra \zeta^{ss}$ means that 
\begin{equation} \label{eqn:convB}
\hat \bV^{+}_i(x) \ra \hat \bV^{ss+}(x), \hat \bV^{0}_i(x) \ra \hat \bV^{ss0}(x), 
\hbox{ and } \hat \bV^{-}_i(x) \ra \hat \bV^{ss-}(x) \hbox{ as } i \ra \infty
\end{equation}
in the respective Grassmann spaces pointwise for each $x \in \Uc \hat M$. 

We may construct a sequence of the fiberwise metrics $\llrrV{\cdot}_{\bR^n_{\rho_i}}$  so that 
the sequence of the lifted norms of $\llrrV{\cdot}_{\bR^n_{\rho_i}}$  
uniformly converging to that of  $\llrrV{\cdot}_{\bR^n_{\rho^{ss}}}$ on $\Uc \hat M \times \bR^n$
over each compact subset of $\Uc \hat M$ 
using a fixed set of a partition of unity subordinate to a covering with trivializations.  

The flow $\phi^{ss}$ on $\bR_{\rho^{ss}}$ satisfies the expansion and contraction properties
in Definition \ref{defn:phyp}. 
Note that the lift $\Phi^{ss}$ of $\phi^{ss}$ and that $\Phi_i$ of the flow $\phi_i$ for $\bR^n_{\rho_i}$ in the product bundle 
$\Uc \hat M \times \bR^n$ are the same by construction
induced from trivial action on the second factor. 
Since $M$ is compact, there exists $t_0$ such that for $t> t_0$, 
\begin{itemize} 
\item $\Phi^{ss}_{t}| \bV^{ss+} $  is expanding by a factor $c^+> 1$ 
\item $\Phi^{ss}_{t}| \bV^{ss-}$ is contracting by a factor $c^- < 1$  with respect to 
the metric $\llrrV{\cdot}_{\bR^n_{\rho^{ss}}}$. 
\end{itemize} 
Hence for sufficiently large $i$, 
the flow $\Phi_i$ on $\bR^n_{\rho_i}$  satisfies the corresponding properties
as well as the domination properties for $\llrrV{\cdot}_{\bR^n_{\rho_i}}$ 
by \eqref{eqn:convB}. 
Hence, $\Phi_i$ is partially hyperbolic with index $k$.  

For converse, suppose that $\rho$ is partially hyperbolic with index $k$. 
Proposition 4.5 of \cite{BPS} shows that $\rho$ is $k$-dominated. 
Lemma \ref{lem:Equiv} shows that $\rho$ is $P_\theta$-Anosov for $k \in \theta$. 
Also, $k < n/2$ since the neutral bundle exists. 
\end{proof}

\part{Partially hyperbolic holonomy and cohomological dimensions}  \label{part1} 

\section{Introduction}

\subsection{Main results} 
This paper continues Part 1 using its notation and terminology. Mainly, we will need Lemma \ref{lem:piUUM}, Definition \ref{defn:phyp}, and Theorem \ref{thm:main2} of Part 1. 

\red{
A well-known conjecture of Auslander is that a closed affine manifold must have a virtually solvable fundamental group. The Auslander conjecture is proved for closed complete affine manifolds of dimension $\leq 3$ by Fried-Goldman \cite{FG83}, for ones with linear holonomy groups in the Lorentz group by Goldman-Kamishima \cite{GK84}, and for ones of dimension $\leq 6$ by Abels-Margulis-Soifer \cite{AMS02}, \cite{AMS05}, \cite{AMS11}, and \cite{AMS97}. In particular, they showed that the linear holonomy group is not Zariski dense in $\SO(k, n-k)$ for $(n-k) - k \geq 2$ in \cite{AMS11}. Their techniques are basically based on a study of Anosov representations.
}


A good strategy is to study this question by investigating group actions. Margulis space-times provide examples (see \cite{CD15}). The existence of properly discontinuous affine actions on $\mathds{A}^n$ for large classes of groups, including all cubulated hyperbolic groups, was discovered by Danciger, Kassel, and Gu\'eritaud in \cite{DGK2020}, where $n$ is somewhat large compared to $\mathrm{cd}(G)$ of the properly acting affine group $G$. There is a survey on this topic in \cite{DDGS}.

We aim to prove: 
\begin{theorem}\label{thm:main-p2} 
Let $N$ be a complete affine manifold  for $n \geq 3$ with the finitely presented fundamental group. 
Suppose that $N$ has a partially hyperbolic linear holonomy group with index $k, k < n/2$,  and $K(\pi_1(N), 1)$ is realized by a finite complex.  

Then the cohomological dimension $\mathrm{cd}(\pi_1(N))$ is $\leq n-k$ for the partial hyperbolicity index $k$ of $\rho$. 
\end{theorem} 
The main idea for proof is that 
we will modify the developing map into a quasi-isometric embedding  into 
a generalized stable affine subspace. 
Hence, each boundary point of the group is associated with an affine subspace.

Recall from Part 1 the set of roots $\theta=\left\{\log \lambda_{i_{1}}-\log \lambda_{i_{1}+1}, \ldots, \log \lambda_{i_{m}}-\log \lambda_{i_{m}+1}\right\}$ 
with $1 \leq i_{1}<\cdots<i_{m} \leq n-1,$ 
of $\mathrm{GL}(n, \bR)$, 
and the parabolic group $P_\theta$. 


Since we can always find FS submanifolds for $\mathds{A}^n/\Gamma$, 
Theorem \ref{thm:main-p2} and Theorem \ref{thm:main2} of Part 1 will imply the result: 
\begin{corollary} \label{cor:main-p2} 
	Let a finitely presented group $G$ acts on $\mathds{A}^n$ , $n \geq 1$, faithfully, properly discontinuously,  and freely. 
	Suppose that $K(G, 1)$ is realized by a finite complex. 
	Suppose that the linear part of $G$ is P-Anosov for a parabolic group $P$.  

Then 
if $P = P_\theta$ for $\theta$ containing $\log \lambda_k-\log \lambda_{k+1}, k \leq n/2$, then $\mathrm{cd}(G) \leq n-k$ and $k < n/2$.  
	
\end{corollary}

When $(n, k) \ne (2, 1), (4, 2), (8, 4), (16, 8)$, 
without the proper action condition, the conclusions of Corollary \ref{cor:main-p2} are also implied by Theorem 1.3 of Canary-Tsouvalas \cite{CT2020} using Corollary 1.4 of Bestvina-Mess \cite{BM91}. The $(2, 1)$-case follows by Benz\'ecri \cite{Benzecri} and Milnor \cite{Milnor}. They work in $\SL_\pm(n, \bR)$; however, the linear part of $G$ can be made into one preserving the $P$-Anosov property.

\red{Under our properness conditions, these cases do not occur since $k < n/2$ holds. Although we have more assumptions, our methods are substantially different and use more direct geometrical arguments of projecting the holonomy cover to a stable affine subspace using coarse geometry.}

Our main point here is that we provide an alternative point of view. Also, we are currently generalizing these to relatively Anosov representations. The theory is currently developing by various groups as stated in the first part of the paper.

We proved the following which supports the Auslander conjecture. 

\begin{corollary}\label{cor:main2-p2} 
	A closed complete affine manifold $M^n$, $n \geq 3$, cannot have a $P$-Anosov linear holonomy group
	for a parabolic subgroup $P$ of $\GL(n, \bR)$. 
	\end{corollary}  
\begin{proof} 
	If otherwise,  $\mathrm{cd}(\Gamma_M) = n \leq n -k $ for any $k$, $1\leq k \leq n/2$ and $k$ in $\theta$ for $P = P_\theta$. 
	\end{proof} 

Again, the corollary is implied by Theorem 1.3 of \cite{CT2020} except for the $(2,1)$-case.
This case is ruled out by Benz\'ecri \cite{Benzecri} or Milnor \cite{Milnor}.

Finally, we obtain some compactness result: 
\begin{corollary} \label{cor:cpt-p2} 
	Suppose that $\rho: \pi_1(N) \ra \GL(n, \bR)$ be a $k$-Anosov representation
	that is a linear part of a properly discontinuous and free affine action on $\mathds{A}^n$, $n \geq 3$. 

Then	there exists a compact collection of affine subspaces of dimension $n-k$ in the affine Grassmannian space
	$\mathcal{AG}_{n-k}(\bR^n)$  invariant under the affine action. 
\end{corollary}

%


%

A well-known conjecture weaker than the Auslander conjecture is that a complete closed affine manifold cannot have a word hyperbolic fundamental group (see \cite{BCL} for a discussion).


We believe that our approach may be a step in the right direction, and plan to generalize this result for relatively Anosov representations, where there are developing series of research (see \cite{kapovich2023relativizing}, \cite{zhu2022p}, \cite{Zhu21}, and \cite{Zhu23}).

\red{However, we have not finished this new research since the field is still developing.}


\subsection{Outline of Part 2}

In Section \ref{sec:preliminary-p2}, 
we show that each affine subspace intersected with $\hat M$ is uniformly contractible. 
We show that the set of complete isometric geodesics in $\hat M$ ending 
at a common point of the ideal boundary $\partial_\infty \hat M$ is $C$-dense in $\hat M$ for some $C> 0$. 
(Note here, a ``geodesic'' for a metric space $X$ is an isometry from a subinterval to $X$. This is not true for Riemannian spaces. Hence, we need to use this notion.) 

We prove Theorem \ref{thm:main-p2} in Sections \ref{sec:neutral-p2}
and \ref{sec:geoconv-p2}:

In Section \ref{sec:neutral-p2}, we will define an affine bundle associated with an FS submanifold $M$ of a closed complete special affine manifold. 
We suppose that we have a partially hyperbolic linear representation. 
Theorem \ref{thm:neutral-p2} will modify the developing section of $\Uc \hat M$ so that
each complete isometric geodesic in $M$ develops inside an affine space in the neutral directions. The modification follows from the idea of 
Goldman-Labourie-Margulis \cite{GLM09}. 
We define $\mathcal{R}_p$ for $p\in \partial_\infty M$ to be the subspace of points on complete isometric geodesics on $\Uc \hat M$ ending at an ideal point $p$.
Proposition \ref{prop:Finalpart-p2} shows that $\mathcal{R}_p$ for each $p\in \partial_\infty M$ always develops into a generalized stable subspace. 
This follows since along the unstable directions, geodesics depart away from one another. 

In Section \ref{sec:geoconv-p2}, we will prove Proposition \ref{prop:quasiiso-p2}   
that $\hat M$ quasi-isometrically embed into generalized stable subspaces
since $\mathcal{R}_p$ embeds quasi-isometrically into one of the subspace, and
$\hat M$ and $\mathcal{R}_p$ are quasi-isometric. 
Then we prove Theorem \ref{thm:main-p2}: We use the quasi-isometric embedding of $\hat M$ into 
a generalized stable affine subspace to show that 
the maximal dimension of the compactly supported cohomology of $\hat M$ is less than 
the dimension of the generalized stable subspace $n-k$.
Since $\mathds{A}^n/\Gamma$ homotopy equivalent to $K(\Gamma, 1)$ 
has an exhaustion by a sequence of  FS submanifolds $M_i$, 
we will obtain the upper bound  $n-k$ of the cohomological dimension of $\Gamma$.

Finally, we prove Corollaries \ref{cor:main-p2}, \ref{cor:main2-p2}, and \ref{cor:cpt-p2}.

\section{ Preliminary} \label{sec:preliminary-p2} 

\red{
The main purpose is to introduce some notation and prove a few elementary propositions not explicitly written up in the literature. These notions are standard in coarse geometry but may not be commonly known by usual geometric topologists. 
Most of the technical background can be found in the book by Dru\c{t}u-Kapovich \cite{DK2018} which surveys the field extensively.}

\subsection{Grassmanians} \label{sub:Haudorff-p2} 
We assume $n \geq 3$ in this article.  
Let $\mathcal{G}_k(\bR^n)$ denote the space of $k$-dimensional subspaces of 
$\bR^n$. 
We consider the space $\mathcal{AG}_k(\bR^n)$ of affine $k$-dimensional subspaces of $\bR^n$.
The space has a proper complete Riemannian metric that we denote by $d_{\mathcal{AG}_k(\bR^n)}$. 
We also use these on subspaces of $\bR^n$ considered as 
$\mathds{A}^n$.

%

\subsection{Metrics and affine subspaces} 
%
%
%
Now,  $\mathds{A}^n=\tilde N$ has an induced complete  $\Gamma$-equivariant Riemannian 
metric from $N=\mathds{A}^n/\Gamma$ to be denoted by $ \dN$. 
Let $d_E$ denote a chosen standard Euclidean metric of $\mathds{A}^n$
fixed for this paper. 
We will assume that $\partial M$ is convex in this paper. 
 Let $d_M$ denote the path metric induced  from a Riemannian metric on 
 $M \subset\mathds{A}^n/\Gamma$, 
and let $d_{\hat M}$ denote the path metric on $\hat M$ induced from it.
 
From now on, we will fix the Rips constant $\delta$ for $\hat M$, and 
assume that $\hat M$ is $\delta$-hyperbolic. 


From Definition 8.27 of \cite{DK2018}, we recall: 
A map $f:X \ra Y$ between two proper metric spaces $(X, d_X)$ and $(Y, d_Y)$ 
is {\em uniformly proper} if $f$ is coarsely Lipschitz and there is a function $\psi: \bR_+ \ra \bR_+$ such that 
\[d_X\hbox{-diam}(f^{-1} (B^{d_Y}(y, R))) < \psi(R) \hbox{ for each }y \in Y, R\in \bR_+.\]
An equivalent condition is that there is a proper continuous function $\eta: \bR_+ \ra \bR_+$ 
so that 
\[ d_Y(f(x), f(y))\geq \eta(d_X(x, y))) \hbox{ for all } x, y \in X. \]
Here, functions satisfying the properties of 
$\psi$ and $\eta$ respectively are called an {\em upper} 
and {\em lower 
distortion functions}. 


A subspace $Y$ in a  metric space $(X, d)$ is {\em uniformly contractible} in a subspace $Y'$, $Y \subset Y'$,  if 
for every  $r> 0$, there exists a real number $R(r) > 0$ depending 
only on $r$ so that $B^d_r(x) \cap Y$ is contractible in $B^d_{R(r)}(x) \cap Y'$
for any $x \in Y$.
(We generalize Block and Weinberger \cite{BW93} and Gromov \cite{Gromov93}.) 

%



For an affine subspace $L$ of $\mathds{A}^n$, we denote by 
$d_L$ the path metric induced from the Riemannian metric of $ \dN$ restricted to $L$. 
We will use $\dN|L$ the restriction of $\dN$ to $L\times L$ as a metric subspace. 

	

\red{The following was not found in any literature.}
Again, this was worked out with Kapovich. 
\begin{theorem}\label{thm:k-conn-p2} 
		Suppose that $M$ is an FS submanifold of a complete affine manifold $N$ covered by $\mathds{A}^n$ with 
an	invariant path metric $\dN$ induced from a Riemannian metric.
	Let $L$ be an affine subspace of $\mathds{A}^n$ of $\dim \leq n$. 
	Let $\hat M \subset \mathds{A}^n=\tilde N$ be the cover of $M$ under the covering map $\mathds{A}^n \ra N$. 

	Then $L \cap \hat M$ is uniformly contractible in $L$ with the metric $d_L$
and $\dN|L$, and $L\cap \hat M$ with metric $d_L$ embeds 
uniformly properly in $\hat M$ with metric $\dN$.
\end{theorem} 
\begin{proof} 
The main idea for proof is to pull everything back to a fundamental domain 
and find a suitable convex ball which is bounded. 

	Let $F$ be a compact fundamental domain of $\hat M \subset \mathds{A}^n$, 
	containing the origin $O$. 
	Let $L'$ be any affine subspace of dimension $\dim L \leq n$.
	Let $d_{L'}$ denote the path metric on $L'$ induced from  $ \dN$. 
	Let $r$ be any positive real number. 
	The $d_{L'}$-ball $B^{d_{L'}}_r(x)$ in $L'$ of radius $r> 0$ 
	for $x \in F$ is a subset of 
	$B^{d_{\hat M}}_r(x)$  for a $d_{\hat M}$-ball of radius $r$ 
	with center $x \in F$ 
	since the endpoints of a $d_{L'}$-path of length $< r$ has 
	$d_{\hat M}$-distances $< r$ from $x$. 
	Since $\bigcup_{x\in F} B^{d_{\hat M}}_r(x)$ is bounded in $d_E$, 
	there is a constant $R(r, F)$ depending only on $r$ and $F$  
	so that $B^{d_{\hat M}}_r(x) \subset B^{d_E}_{R(r, F)}(O)$ for every $x \in F$
	for the Euclidean ball $B^{d_E}_{R(r, F)}(O)$ of radius 
	$R(r, F)$ with center $O$. 
	
	We take $C(R, F)$ for each $R> 0$ to be the supremum of 
	\[\{d_{L'}(x, y)| x \in F \cap L', y \in  B^{d_E}_R(O)\cap L'\}\]
where $L'$ varies over the collection of affine subspace $L'$ with $\dim L'=\dim L$
and $L' \cap F \ne \emp$. 
Since the set of such subspaces, $F$,  and $\clo(B^{d_E}_R(O))$ are compact, 
and $d_{L'}(x, y)$ is a continuous function of $L'$ and $x, y$,
the supremum exists. 
	Now, \[B^{d_E}_R(O)\cap L' \subset B^{d_{L'}}_{C(R, F)}(x)\subset L'
	\hbox{ for } x\in F \cap L'\] 
and for any affine subspace $L'$ with $\dim L'=\dim L$ containing $x \in F$. 
	Now, $B^{d_E}_R(O) \cap L'$ is convex
	and is a subset of $B^{d_{L'}}_{C(R, F)}(x)$. 	
	Since \[B^{d_{L'}}_r(x) \subset B^{d_E}_{R(r, F)}(O) \cap L', x\in F,\]  
	$B^{d_{L'}}_r(x)$ is contractible to a point 
	inside $B^{d_{L'}}_{C(R(r, F), F)}(x) \subset L'$.
	
	Since we can put any  $B^{d_{L}}_r(x)$ for $x\in L \cap \hat M$ to 
	a $d_{\gamma(L)}$-ball with the center in $F$ 
	by a deck transformation $\gamma$ of $\hat M$, 
	we obtained the radius $C(R(r, F), F)$ for each $r> 0$ so that 
	the uniform contractibility holds.  
	%
	%
	%
	%
	%

We can do the same for $\dN|L$ by doing the same argument as  above. 

The final uniform properness of the embedding $(L\cap \hat M, d_L) \ra (\hat M, \dN)$  is
proved: 
Given two points $x$ and $y$ of $L\cap \hat M$ of 
$\dN$-distance $< r$, they can be homotopied to a point 
in a $d_L$-ball of radius $R(r)$ for some $R(r)$ in $L$ by the part on $\dN|L$ 
in the above paragraph. Hence, $d_L(x, y) < 2R(r)$. 
%
\end{proof}

\subsection{Cobounded map and parallel homotopy}

Let  $(Z, d_Z)$ and $(Y, d_Y)$ be proper geodesic metric spaces. 
If $Y \subset Z$, then 
a function $f:Y \ra Z$ is {\em $d_Z$-cobounded} if 
$d_Z(x, f(x)) < C$ for a constant independent of $x$. 

A homotopy $H:Y \times I \ra Z$ is {\em $d_Z$-parallel} if 
$d_Z(H(z, t), z) \leq C$ for a constant $C$ independent of $z, t$. 

We will use the metric $\dN$ for these. 

\begin{lemma} \label{lem:parallel-p2} 
Let $f_i: Y \ra \mathds{A}^n$ be two maps where 
$ \dN(f_1(y), f_2(y)) \leq C$. 
Then $f_1$ and $f_2$ are $\dN$-parallelly homotopic. 
In particular, a $\dN$-cobounded map $Y \ra \tilde N=\mathds{A}^n$ is $\dN$-parallelly homotopic to the 
inclusion $Y \ra \mathds{A}^n$.
\end{lemma} 
\begin{proof} 
We define the homotopy 
$H(y, t) = t f_1(y) + (1-t) f_2(y)$ for $y\in Y, t \in [0, 1]$. 
For a fixed $y$, the $ \dN$-path length is bounded above by a constant $C'$ by our premise and Theorem \ref{thm:k-conn-p2}. 
Hence, $H$ is a $\dN$-parallel homotopy. 
The second part is immediate. 
\end{proof}

\subsection{The $C$-density of geodesics} 
\red{
The main purpose of this subsection is to prove Theorem  \ref{thm:Ucdense-p2} that the set of all geodesics ending at an ideal point is ``roughly'' dense in $\hat M$. 
 As far as the author knows, this elementary result was not found in literature. }

A subset $A$ of $\hat M$ is {\em $C$-dense in $\hat M$} for $C> 0$ 
if $d_{\hat M}(x, A )< C$ for every point $x\in \hat M$. 

\begin{lemma} \label{lem:geotwoends-p2} 
A geodesic in a Gromov hyperbolic space $X$  has two distinct endpoints in 
$\partial_\infty X$. 
\end{lemma} 
\begin{proof} 
Rays in a geodesic in different directions cannot be asymptotic
since the geodesic is isometrically embedded. 
(See Section 3.11.3 of \cite{DK2018}.)
\end{proof}

 We call constant $C$ satisfying the conclusion  
below the {\em quasi-geodesic constant}. 

\begin{lemma} \label{lem:nonasym-p2} 
	Given two rays $m$ and $m'$ ending at $p$ and $q$ in $\partial_\infty \hat M$. If $p\ne q$, then $d_{\hat M}(m(t), m'(t))\ra \infty$ 
	as $t \ra \infty$. 
	\end{lemma} 
\begin{proof} 
	Suppose that $d_{\hat M}(m(t_i), m'(t_i))$ is bounded for 
	some sequence $t_i$ with $t_i \ra \infty$. By
	Theorem 1.3 of Chapter 3 of \cite{CDP90}, $d_{\hat M}(m(t), m'(t))$ is uniformly bounded 
	since $m'(t)$ follows $m(t)$ as a quasi-geodesic. 
		If $d_{\hat M}(m(t), m'(t))$ is bounded, then $p = q$. 
Hence, the only possibility is that
 $d_{\hat M}(m(t), m'(t))\ra \infty$ as $t\ra \infty$. 
	\end{proof} 

\begin{lemma} \label{lem:uniform-nonasym-p2}
	Let $p$ be a point of $\hat M$. Let $B_1 \subset \partial_\infty \hat M$ 
and $B_2 \subset \hat M$ be two disjoint compact subsets. 
	Consider the set $S_{B_1, R}$, $i=1, 2$, be the set of points on rays from $p$ ending in $B_1$ outside the ball $B^{d_{\hat M}}_R(p))$.
	Then $d_{\hat M}(S_{B_1, R}, B_2) \ra \infty$ as $R \ra \infty$. 
	\end{lemma} 
\begin{proof} 
	Suppose not. 
Then there exists a sequence of points $y_j = \gamma^{(i)}_j(t_i) \in S_{B_1, R}$ for a geodesic 
$\gamma_j$ ending in $v_i \in B_1$ starting from $p$
and a sequence $z_i \in  B_2$
where $d_{\hat M}(y_j, z_j)$ is bounded above by a constant $C$ and $t_{i} \ra \infty$. 
Since $d_{\hat M}(z_i, p)$ is bounded above, 
$d_{\hat M}(p, y_i)$ is bounded above. This is a contradiction since 
$d_{\hat M}(p, y_i)= t_i$.  
	\end{proof} 

\begin{figure}[ht!]
	\labellist
	\small\hair 2pt
	\pinlabel $y$ [l] at 248 263
	\pinlabel $l_i(0)$ [l] at 200 247
	\pinlabel $m_i$ [l] at 250 200
	\pinlabel $l_i$ [r] at 200 200
	\pinlabel $l_i(t)$ [r] at 233 80
	\pinlabel $m_i(t'_i(t))$ [l] at 255 90
	\pinlabel $q_i$ [r] at 253 42
	\pinlabel $r_i$ [t] at 110 470
	\endlabellist
	\includegraphics[width=0.6\textwidth]{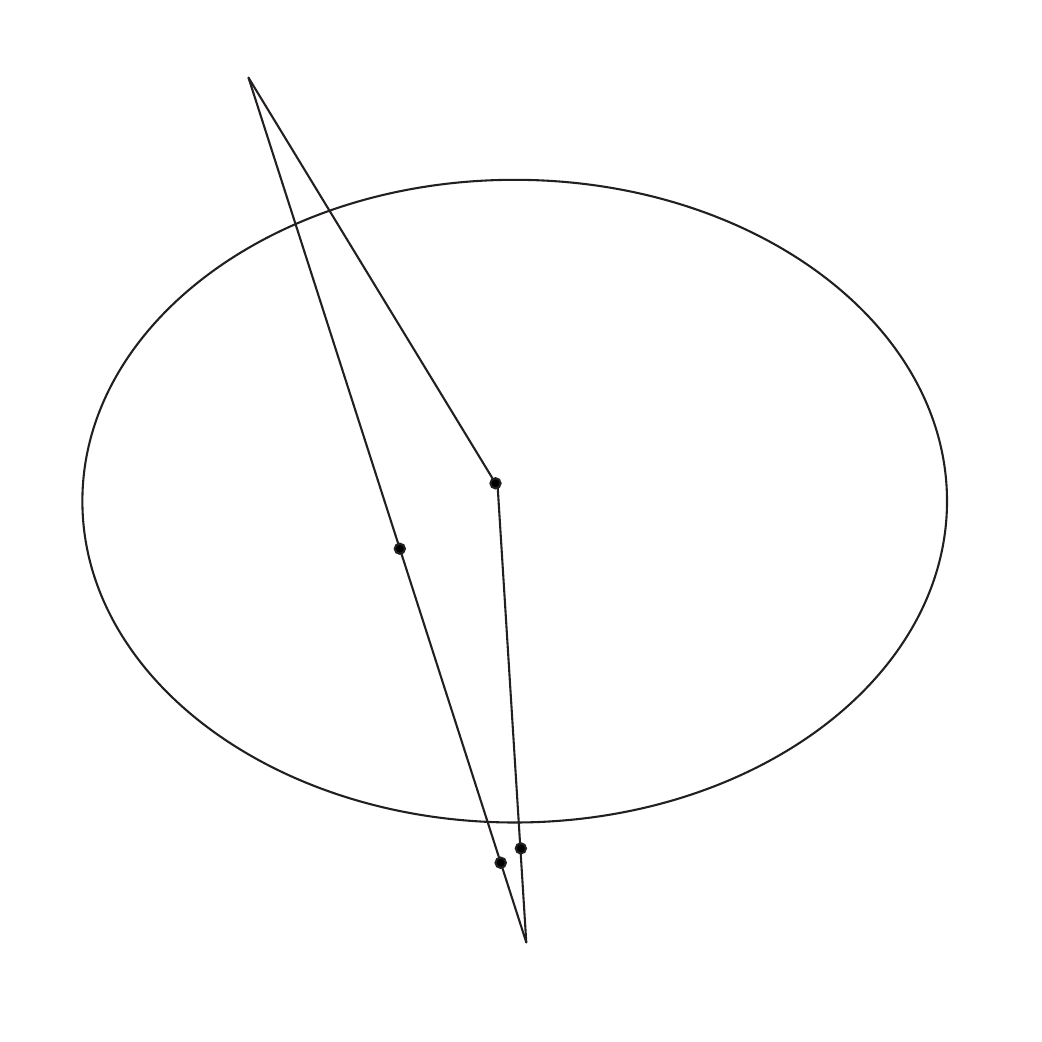}
	\caption[The proof of Lemma \ref{lem:geolim-p2}]{The proof of Lemma \ref{lem:geolim-p2} }
	\label{fig:circle-p2}
\end{figure}

The author cannot find the following elementary lemma in the literature. 
\begin{lemma}\label{lem:geolim-p2} 
	Let $q_i$ and $r_i$ be the forward and backward endpoints respectively in 
	$\partial_\infty \hat M$ of a complete isometric geodesic
	$l_i$. 
	Suppose that $l_i \ra l$ for a complete isometric geodesic $l$. 
	Suppose $q_i \ra q$ and $r_i \ra r$ for $q, r\in \partial_\infty \hat M$. 

	Then $l$ has endpoints $q$ and $r$.
	\end{lemma} 
\begin{proof} 
	Choose a point $y \in l$ and let 
	$m_i$ be a ray from $y$ to $q_i$
	as obtainable by Proposition 2.1 of Chapter 2 of 
	\cite{CDP90}. 
Let $B$ be a compact subset of $\partial \hat M$ containing all $q_i$
and $B'$ be another one containing all $r_i$ which is disjoint from $B$. 
We can choose such sets by Section 11.11 of \cite{DK2018} using the Shadow topology
since $\partial \hat M$ is first-countable.
We may assume without loss of generality that $l_i(0) \ra y$. 
Let $K$ be the \red{compact} convex hull of the compact set containing all $l_i(0)$ and $y$
and the isometric geodesic segments with endpoints $l_i(0)$ and $y$.  
Let $R_0$ be the number 
so that $d_{\hat M}(S_{B, R_0}, K) \geq 24 \delta +1$ by
Lemma \ref{lem:uniform-nonasym-p2} , and 
let $R= \max\{R_0, d_{\hat M}(y, l_i(0))|i=1, 2, \dots\}$.
%
	
	
	Considering the geodesic triangles with vertices $l_i(0), q_i, y$
	and with two edges equal to $m_i$ and a part of $l_i$ from $l_i(0)$, 
	we obtain a function $t'_i$ with values $> R$ where
		 \begin{equation}\label{eqn:mi1-p2} 
		d_{\hat M}(l_i(t), m_i(t'_i(t))) \leq 24\delta
		\hbox{ for }  t> 0 \hbox{ provided } l_i(t) \not\in B^{d_{\hat M}}_{R+24\delta}(y)
	\end{equation} 
	 by the $\delta$-hyperbolicity of $\hat M$,  
	and Proposition 2.2 of Chapter 2 of \cite{CDP90}.  
	
   Let $t_{i,0}$ be the last time when $l_i(t)$ leaves the ball 
$\partial B^{d_{\hat M}}_{R+24\delta}(y)$.
Then the function $t'$ is defined on $[t_{i, 0}, \infty)$. 
		Moreover, we obtain $0 \leq t_{i, 0} \leq 2R + 24 \delta $ by using three points $y, l_i(0)\in 
	B^{d_{\hat M}}_{R}(y)$ and $l_i(t_{i, 0}) \in \partial  B^{d_{\hat M}}_{R+24\delta}(y)$
	and the triangle inequality. 
	Hence, the function $t'_i$ is always defined on $[2R + 24 \delta, \infty)$. 
	
	Now, $R \leq t'_i(t_{i, 0}) \leq R + 48 \delta$ by the condition \eqref{eqn:mi1-p2} 
	and the triangle inequality. 
	Since $d_{\hat M}( l_i(t), m_i(t'_i(t)))$ is within 
	$24\delta$, and $m_i$ is also an isometry,  
	we obtain 
	\begin{multline}
	 (t-t_{i, 0}) \leq d_{\hat M}(l_i(t_{i, 0}), y)+ 
d_{\hat M}(y, m_i(t'_i(t))) + d_{\hat M}(m_i(t'(t)), l_i(t))
\leq \\
(R+  24 \delta) + t'_i(t) + 24 \delta,  
	\end{multline}
by applying the triangle equalities to the chain of four points 
$l_{i}(t_{i, 0})$, $y$, $m_i(t'_i(t))$, and $l_i(t)$.
And 
by applying the triangle equalities to the chain of four points 
 $m_i(t'_i(t))$, $l_i(t)$, $l_{i}(t_{i, 0})$, and $y$, we obtain
 \begin{multline}  \label{eqn:mi2-p2} 
	t'_i(t) \leq 
d_{\hat M}(m_i(t'_i(t)), l_i(t)) + 
d_{\hat M}(l_i(t), l_i(t_{i, 0})) +
d_{\hat M}(l_i(t_{i,0}), y)  \\  \leq 
24 \delta + 
(t- t_{i, 0})+ R + 28\delta.
\end{multline} 
Combining the two, we obtain 
\[	 (t-t_{i, 0}) -R - 48 \delta \leq t'_i(t) 
	 \leq  (t-t_{i,0}) + R + 48\delta.
\]



By choice of a subsequence, 
we may assume 
$m_i$ converges to a ray $m$ from $x$ to $q$ 
since $\partial_\infty \hat M$ has the shadow topology. 
	(See Section 11.11 of \cite{DK2018}.)
 By \eqref{eqn:mi2-p2} and the Arzel\`a-Ascoli theorem, we may assume 
$t'_i(t) \ra t'(t)$ for $t\in [2R + 24 \delta, \infty)$ and $t_{i, 0} \ra t_0, t_0\in [0, 2R + 24 \delta]$ 
\red{for a function $t'$ defined on $[2R + 24 \delta, \infty)$}
up to a choice of a subsequence. 
(See Section \ref{sub:geoconv}.)  
	Hence, we obtain by \eqref{eqn:mi1-p2} 
	\[d_{\hat M}(l(t), m(t'(t))) \leq 24\delta\] 
	for $t\in [2R + 24 \delta, \infty)$ 
	Hence, $l$ ends at $q$ as $t \ra \infty$.
	
	Similarly, we can show that $l$ ends at $r$
	as $t\ra -\infty$.  
	\end{proof}


Let $X$ be a first countable Hausdorff space. 
Recall that a {\em  lower semi-continuous} function $f: X \ra \bR_+$ is 
a function satisfying
$f(x_0) \leq \liminf_{x \ra x_0} f(x)$ for each $x_0\in \hat M$. 
A lower semi-continuous function always achieves an infimum. 
(See \cite{Vinogradova} for details.)
Let $C > 0$. 
A function $f$ is {\em $C$-roughly continuous} if 
\[|\liminf_{x\ra x_0} f(x) - f(x_0)| \hbox{ and }
|\limsup_{x\ra x_0} f(x) - f(x_0)| < C \hbox{ for all } x_0 \in \hat M\] 
If $f$ is lower semi-continuous and satisfies 
$\limsup_{x\ra x_0} f(x) <  f(x_0) + C$ for all $x_0 \in \hat M$,
then it is $C$-continuous  since 
\[f(x_0) \leq \liminf_{x_i \ra x_0} f(x) \leq 
\limsup_{x_i \ra x_0} f(x)\] holds.  

\label{page-semiC}

			Let $p$ be a point of the ideal boundary $\partial_\infty \hat M$. 	
	We defined $\mathcal{R}_p$ to be the union of complete isometric geodesics
	in $\Uc \hat M$ mapping to complete isometric geodesics
	in $\hat M$ ending at $p$. 
	A {\em geodesic of $\mathcal{R}_p$} is one of these geodesics
	in $\Uc \hat M$ or $\hat M$. 
	Define a function \[f_q: \hat M \ra \bR_+ \hbox{ given by } 
	f_q(x) := d_{\hat M}\left(x, \pi_{\Uu \hat M} \left(\bigcup \mathcal{R}_q\right)\right), x \in \hat M.\] 
		
		Let $q \in \partial_\infty \hat M$. 
		The set of complete isometric geodesics ending at $q$ and passing a compact subset of
	$\hat M$ is closed under the convergences. 
	(See Section \ref{sub:geoconv}  of Part 1.) 
	A complete isometric geodesic $l$ realizes $f_q(x)$ for each $x \in \hat M$. 
	That is, for each $x$ in $\hat M$, there is a complete isometric 
	embedded geodesic $l$ in $\mathcal{R}_q$ where 
	$d_{\hat M}(x, y)$ for $y\in \pi_{\Uu \hat M} (l)$ 
	realizes the infimum.

	\begin{lemma} \label{lem:semicont-p2} 
   $f_q(x)$ is a lower semi-continuous function of $q$ and $x$ respectively. 
	\end{lemma} 
\begin{proof} 
Since $\partial \hat M$ is first countable by Section 11.11 of \cite{DK2018}, 
we need to worry about sequences only. 
Let $q_i$, $q_i\in \partial_\infty M$, be a sequence converging to $q$. 
Then $f_{q_i}(x)$ equals $d_{\hat M}(x, l_i)$ for a complete isometric geodesic $l_i$ ending at $q_i$. 
Since $l_i$ has a distance from $x$ bounded from above, it has 
a limiting geodesic $l_\infty$ up to a choice of subsequences. 
(See Section \ref{sub:geoconv}   of Part 1.)  
Since we have $l_i(t) \ra l_\infty(t)$ for each $t \in \bR$, we obtain 
\begin{equation}\label{eqn:linfty} 
\liminf_{i \ra \infty} f_{q_i}(x)= d_{\hat M}(x, l_\infty). 
\end{equation} 
 
By Lemma \ref{lem:geolim-p2}, $l_\infty$ ends at $q$. 
%
$l_\infty$ lifts to a geodesic in $\mathcal{R}_q$. 
Since $f_q(x) = d_{\hat M}(x, l)$ for some geodesic $l$ ending at $q$, 
and is the infimum value for all geodesics $l'$ in $\mathcal{R}_q$, 
$\liminf_{i\ra\infty} f_{q_i}(x) \geq f_q(x)$ by 
\eqref{eqn:linfty}. 

We can prove 
the lower semi-continuity with respect to $x$ similarly.  
	\end{proof}

\begin{figure}[ht!]
	\labellist
	\small\hair 2pt
	\pinlabel $y$ [l] at 245 255
	\pinlabel $x$ [r] at 240 255
	\pinlabel $q'$ [r]  at 265 450
	\pinlabel $\partial B_R^{d_{\hat M}}(y)$ [r] at 415 250
	\pinlabel $l_i$ [r] at 230 300
	\pinlabel $l$ [l] at 265 300
	\pinlabel $l'_i$ [l] at 220 190
	\pinlabel $l'''$ [l] at 290 115
	\pinlabel $z_i$ [r] at 220 105
	\pinlabel $z$ [l] at 265 98
	\pinlabel $q_i$ [r] at 187 49
	\pinlabel $q$ [l]  at 250 36
	\endlabellist
	\includegraphics[width=0.90\textwidth]{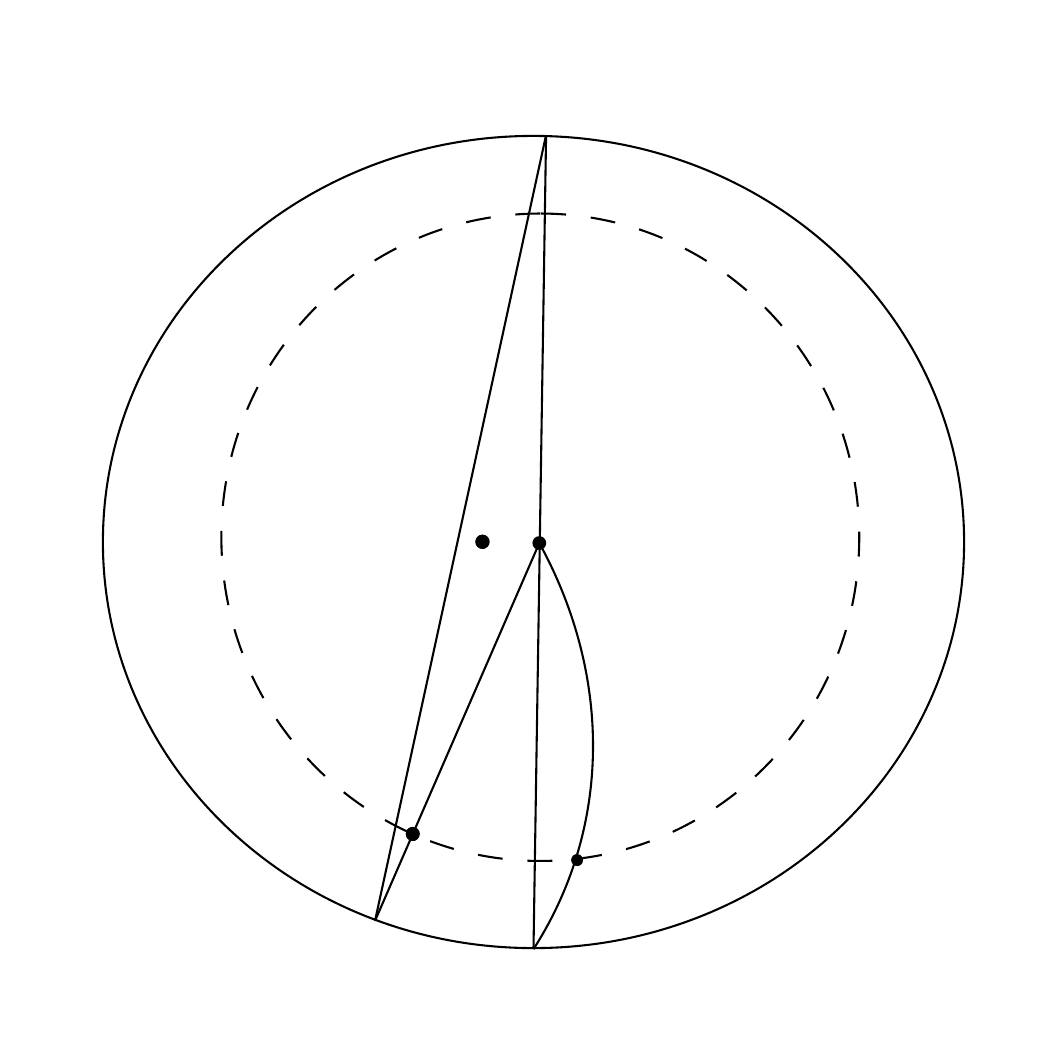}
	\caption[The proof of Lemma \ref{lem:Crough-p2}]{The proof of Lemma \ref{lem:Crough-p2} }
	\label{fig:converge-p2}
\end{figure}

Recall the quasi-geodesic constant from Lemma \ref{lem:twogeo}.  
	
	\begin{lemma}\label{lem:Crough-p2} 
	Let $C'$ be the quasi-geodesic constant. Let $x\in \hat M$. 
	Then $f_q(x)$ is a $C'$-roughly continuous function of $q$.
	\end{lemma} 
\begin{proof} 
Suppose that $q_i \in \partial_\infty \hat M$ converges to $q \in \partial_\infty \hat M$. Clearly, $\liminf_{i \to \infty} f_{q_i}(x) \geq f_q(x)$ since any isometric geodesic $l_i$ with endpoints $q_i$ within a $d_{\hat M}$-distance $\leq C''$ from $x$ for some $C'' > 0$ converges to an isometric geodesic ending at $q$ by the Arzel\`a-Ascoli theorem (see Section \ref{sub:geoconv}).


Let $l$ be as above, realizing $f_q(x)$, which is a complete isometric geodesic with endpoints $q$ and $q'$ in $\partial_\infty \hat M$. By Lemma \ref{lem:geotwoends-p2}, we have $q \ne q'$. We can find a complete isometric geodesic $l_i$ with endpoints $q_i \in \partial_\infty \hat M$ and $q'$ by Proposition 2.1 of Chapter 2 of \cite{CDP90}, where $q_i \to q$.

5	Let $l$ be as above realizing $f_q(x)$ which is 

	We claim that $l_i$ meets a fixed compact subset of $\hat M$: 
	We take a point $y$ on $l$ so that $d_{\hat M}(x, y) < f_q(x) + 1$. 
	Then we take an isometric geodesic $l'_i$ from $y$ to $q_i$ by Proposition 2.1 of Chapter 2 of \cite{CDP90}.
	Let $l''$ be a ray in $l$ from $y$ to $q'$. 
	By taking a subsequence, we obtain $l'_i \ra l'''$ to a ray $l'''$ from $y$. 
	Again, $l'''$ ends at $q$ by the shadow topology. 
Now, $l'_i$ is in a $24\delta$-neighborhood of $l'' \cup l_i$ by Proposition 2.2 of 
\cite{CDP90}. 

Since $q$ and $q'$ are distinct, the respective rays from $y$ ending 
at $q$ and $q'$ do not have a bounded Hausdorff distance 
by Lemma \ref{lem:nonasym-p2}. 
Let $R$ be a large number so that 
\begin{itemize} 
\item $\partial B^{d_{\hat M}}_R(y) \cap (l''' - N_{24\delta}(l''))$ 
contains a point $z$, and 
\item  $B^{d_{\hat M}}_\eps(z)$ is disjoint from $N_{24\delta}(l'')$
for sufficiently small $\eps$, $\eps> 0$. 
\end{itemize} 
For sufficiently large $i$, there is a sequence $z_i$ for 
 $z_i \in l'_i \cap \partial B^{d_{\hat M}}_R(y)$ where $z_i \ra z$. 
Hence, $z_i \not\in N_{24\delta}(l'')$ for sufficiently large $i$. 
Then $z_i$ is in a $24\delta$-neighborhood of $l_i$ by the
conclusion of the above paragraph. 
Hence, we obtain $d_{\hat M}(\partial B^{d_{\hat M}}_R(y), l_i) \leq 24 \delta$ and 
$l_i $ meets $B^{d_{\hat M}}_{R+ 24\delta +1}(y)$ for sufficiently large $i$. 

	Therefore, the sequence of $l_i$ reparametrized with 
$l_i(0) \in B^{d_{\hat M}}_{R+ 24\delta +1}(y)$ 
	converges to a complete isometric
	geodesic $l'$ with the same endpoints as $l$ up to a choice of a subsequence $j_i$ 
	by Lemma \ref{lem:geolim-p2}. 
Because of 
\begin{itemize} 
\item $f_{q_i}(x) \leq d_{\hat M}(l_i, x)$ and 
\item  the fact that  
\begin{equation} 
d_{\hat M}(l_{j_i}, x) \ra d_{\hat M}(l', x) \leq  d_{\hat M}(l, x) + C'
\hbox{ implies } 
\limsup_{i\ra \infty} f_{q_i}(x) \leq f_q(x) +  C' 
\end{equation}
by Lemma \ref{lem:twogeo}  of Part 1, 
\end{itemize} 
	we obtain $\limsup_{i \ra \infty} f_{q_i}(x) \leq f_q(x) + C'$.  	
%
Lemma \ref{lem:semicont-p2} completes the proof
by the fact in Page \pageref{page-semiC}. 
\end{proof}

	The set $\bigcup_{q\in \partial_\infty \hat M} \bigcup \mathcal{R}_q$ is a closed set in $\Uu \hat M$ 
since it equals $\Uc \hat M$. 

	\begin{theorem}[Horospherical geodesic rough density] \label{thm:Ucdense-p2} 
		Let $M$ be a compact manifold with a covering map $\hat M \ra M$
		with a deck transformation group $\Gamma_M$. 
		Suppose that $\Gamma_M$ is word-hyperbolic.
		Let $p \in \partial_\infty \hat M$. 

		Then every point $x$ of $\hat M$ is in a bounded distance from 
		a complete geodesic of $\mathcal{R}_p$ for a constant $C, C> 0$, and
		$\pi_{\Uu \hat M} (\mathcal{R}_p)$ is $C$-dense in $\hat M$. 
	\end{theorem} 
	\begin{proof} 
	For each $x \in \hat M$, we claim that $f_q(x) \leq C_x$ for every $q$ 
	for a constant $C_x > 0$ 
	depending on $x$ since $\partial_\infty \hat M$ is compact: 
	If not, we can find a sequence $q_i$ in $\partial_\infty \hat M$ 
	so that 
	a sequence of rays $r_i$ from $x_0$ to $q_i$ 
	converges to a ray $r_\infty$ from $x_0$ to a point 
	$q_\infty$ of $\partial_\infty \hat M$  
	so that $f_{q_i}(x) \ra \infty$. 
	(See Lemma 11.77 of \cite{DK2018}.)
 We have a contradiction by Lemma \ref{lem:Crough-p2} 
 since $f_{q_\infty}(x)$ is finite.  
	
	We define 
	$f: \hat M \ra \bR_+$ by 
	$f(x) = \sup_{q \in \partial_\infty \hat M} f_q(x)$. 
   Since $f_q$ is lower semi-continuous function of $x$ as well, 
   $f$ is a lower semi-continuous function of $x$
	by the standard theory. (See \cite{Vinogradova}.)
	Since $\Gamma_M$ acts on $\partial_\infty \hat M$, 
	$f$ is $\Gamma_M$-invariant. 
	
Now, $f$ induces a lower semi-continuous function $f': M \ra \bR_+$.
Since $f'$ is lower semi-continuous, 
there is a minimum point $x_0 \in \hat M$ under $f$. 

In other words, 
for $x_0$, $f_{q'}(x_0) < C'$ for a constant $C'> 0$ independent of $q'$, $q'\in \partial_\infty \hat M$.  
Hence, $f_q(\gamma(x_0))= f_{\gamma^{-1}(q)}(x_0) < C'$ for any $\gamma \in \Gamma_M$. 
For every point $x$ in $\hat M$, 
$f_q(x) \leq f_q(\gamma(x_0)) + d_{\hat M}(x, \gamma(x_0))$
by the triangle inequality. 
Since a choice of $\gamma$ can bound the second term, 
it follows that $f_q(x) < C''$ for a constant $C''> 0$
for every $x \in \hat M$. 
%
%
%
%
%
\end{proof} 

We remark that we cannot find this type of elementary results in the literature.

\section{Decomposition of the vector bundle over $M$ 
	and sections of the affine bundle. } \label{sec:neutral-p2} 

\red{For the main idea here, see the outline in the introduction. }



\subsection{Modifying the developing sections} 
Let $M$ be an FS submanifold of closed complete affine manifold $N$
with a cover $\hat M \subset \mathds{A}^n$.
We assume $\partial M$ is convex. 
$N$ has the developing map $\dev: \tilde N \ra \mathds{A}^n$, which we may consider as the identity map. 
There is restricted developing map $\dev: \hat M \ra \mathds{A}^n$. 
We may consider this as the inclusion map.
Let $\rho':\Gamma_M \ra \Aff(\mathds{A}^n)$ denote the associated affine 
holonomy homomorphism. Let $\Gamma$ denote the image. 

There is a covering map $\hat M \ra M$ inducing 
the covering map $p: \Uu \hat M \ra \Uu M$. 
The deck transformation group equals $\Gamma_M$. 

We form $\mathds{A}^n_{\rho'}$ as the quotient space of 
 $\Uc \hat M \times \mathds{A}^n$ and $\Gamma_M$ acts by the  action
 twisted by $\rho'$ 
 \[\gamma ((x, \vec{v}), y)= ((\gamma(x), D\gamma(\vec{v})), \rho'(\gamma)(y)) \hbox{ for } \gamma \in \Gamma_M\]
 for the map $D\gamma:\Uu M \ra \Uu M$ induced by the differential of $\gamma$. 
We have a projection $\hat \Pi_{\mathds{A}^n}: \Uc \hat M \times \mathds{A}^n \ra \mathds{A}^n$ 
inducing \[\Pi_{\mathds{A}^n}: (\Uc \hat M \times \mathds{A}^n)/ \Gamma_M \ra \mathds{A}^n/\Gamma, \]
and another one $\hat p_{\Uc M}: \Uc \hat M \times \mathds{A}^n \ra \Uc \hat M$ 
inducing 
\begin{equation}\label{eqn:piUM}  
p_{\Uc M}: (\Uc \hat M\times \mathds{A}^n)/\Gamma_M \ra \Uc M.
\end{equation}
We define a section
$\hat s: \Uc \hat M \ra \Uc \hat M \times \mathds{A}^n$ where
\begin{equation}\label{eqn:defines-p2}  
\hat s((x, \vec{v}))= ((x, \vec{v}), \dev(x)), (x, \vec{v}) \in \Uc \hat M. 
\end{equation} 
Since \[\hat s(g(x, \vec{v})) = (g(x, \vec{v}), \rho'(g)\circ \dev(x))
\hbox{ for } (x, \vec{v})\in \Uc \hat M, g \in \Gamma_M,\] 
$\hat s$ induces a section $s: \Uc M \ra \mathds{A}^n_{\rho'}$.
We call $s$ the {\em section induced by a developing map}. 
(See Goldman \cite{G88})


There is a flat connection $\hat \nabla$ on the fiber bundle $\Uc \hat M \times \mathds{A}^n$ 
over $\Uc \hat M$ induced from the product structure.
This induces a flat connection $\nabla$ on $\mathds{A}^n_{\rho'}$. 
Let $V_\phi$ denote 
the vector field on $\Uc M$ along the geodesic flow $\phi$ 
of $\Uu M$. 
The space of fiberwise vectors on  
$\Uc \hat M \times \mathds{A}^n$ 
equals $\Uc \hat M \times \bR^n$.
Hence, the vector bundle associated with the affine bundle 
$\mathds{A}^n_{\rho'}$ is $\bR^n_{\rho}$. 
Let 
 $\llrrV{\cdot}_{\mathds{A}^n_{\rho'}}$ 
 denote the fiberwise metric 
 induced from $\llrrV{\cdot}_{\bR^n_\rho}$. 
Now $\Uu \hat M$ have the Riemannian metric $d_{\Uu \hat M}$ invariant 
under the action $\Gamma_M$. 

\begin{itemize} 
	\item Let $d_{\mathrm{fiber}}$ denote the fiberwise distance metric on 
	$\Uc \hat M \times \mathds{A}^n$ 
	from the fiberwise norm $\llrrV{\cdot}_{\mathds{A}^n_{\rho'}}$.
	\item Let $d_{\tilde N, {\mathrm{fiber}}}$ 
	denote the fiberwise distance metric on
	$\Uc \hat M\times \mathds{A}^n$ with the second factor given the metric $ \dN$. 
\end{itemize} 
Both fiberwise metrics 
are invariant by the $\Gamma_M$-action twisted with $\rho'$.  




\begin{theorem} \label{thm:neutral-p2} 
	Let $M$ be an FS submanifold of a complete affine manifold 
	with convex $\partial M$. 
	Suppose that $\Gamma_M$ is word-hyperbolic. 
	Suppose that $M$ has a partially hyperbolic linear holonomy homomorphism with respect to a Riemannian metric on $M$ in the bundle sense.  

Then there is a section $s_\infty$ homotopic to the developing section $s$ in the $C^{0}$-topology
with the following conditions\/{\rm :} 
\begin{itemize} 
\item $\nabla_{V_\phi} s_\infty(x)$ is in $\bV_0(x)$ for each $x \in \Uc M$.
\item $d_{\tilde N_{\rho'}}(s(x), s_{\infty}(x))$ is uniformly bounded for every $x \in \Uc M$
with respect to $d_{\mathrm{fiber}}$. 
\item $ \dN(\hat \Pi_{\mathds{A}^n}\circ \hat s(x), \hat \Pi_{\mathds{A}^n}\circ \hat s_{\infty}(x))$ is uniformly bounded for $x \in \Uc \hat M$ with respect to $\dN$. 
\item $\hat \Pi_{\mathds{A}^n} \circ \hat s_\infty: \Uc \hat M \ra \mathds{A}^n$ is $\dN$-parallelly homotopic to $\dev\circ \pi_{\Uu \hat M} $
 for the metric $ \dN$. 
\item $\hat \Pi_{\mathds{A}^n} \circ \hat s_\infty:\Uc \hat M \ra \mathds{A}^n$ is 
a quasi-isometric equivariant embedding 
with respect to $d_{\Uu \hat M}$ and $ \dN$. 
\end{itemize} 
\end{theorem} 
\begin{proof} 
We define as in \cite{GLM09} 
\[s_\infty := s + \int^\infty_0 (D\Phi_t)_\ast (\nabla^-_{V_\phi} s) dt - 
\int^\infty_0 (D\Phi_{-t})_\ast (\nabla^+_{V_\phi} s) dt.\]  
These integrals are  bounded in $\llrrV{\cdot}_{\bR^n_\rho}$ 
since the integrands are exponentially decreasing 
in the fiberwise metric at $t\ra \infty$. (See Definition \ref{defn:phyp} of 
Part 1.) 
Then it is homotopic to $s$ since we can replace $\infty$ by $T, T> 0$
and let $T \ra \infty$. 
Also $\nabla_{V_\phi}(s_\infty) \in \bV_0$ as in 
the proof of Lemma 8.4 of \cite{GLM09}.
The continuity of $s_\infty$ follows since we have exponential decreasing sums.
This proves the first two items. 

Let $F$ denote a compact fundamental domain of $\Uu \hat M$. 
Since the image of $\hat s(F)\cup \hat s_{\infty}(F)$ is 
a compact subset of $\mathds{A}^n_{\rho'}$, we obtain 
\[d_{{\mathrm{fiber}}}(\hat s(x), \hat s_\infty(x)) < C', 
x\in F\cap \Uc\hat M  \hbox{ for a constant } C'.\] 
By the $\Gamma_M$-invariance, we obtain 
\begin{equation} \label{eqn:tildes} 
d_{{\mathrm{fiber}}}(\hat s(x), \hat s_\infty(x)) < C'
\hbox{ for }x\in \Uc \hat M. 
\end{equation} 
For some uniform constant $C''$ and $C'''$,
we have
\begin{equation}\label{eqn:ssinfty}
 \dN(\hat \Pi_{\mathds{A}^n} \circ\hat s(x), \hat \Pi_{\mathds{A}^n} \circ\hat s_\infty(x))
\leq 
C'' d_{{\mathrm{fiber}}}(\hat s(x), \hat s_\infty(x)) + C''', 
x\in \Uc\hat M. 
\end{equation}
This follows since we can consider $x$ to be in a compact fundamental domain 
$F$ of $\hat M$ and $\hat s$ is bounded on $F$, and we use the equivariance of 
the metrics $\dN$ and $d_{\mathrm{fiber}}$.  
The third item follows from this. 
The fourth item follows by Lemma \ref{lem:parallel-p2}.


The final item follows since $s_\infty$ is a continuous map: 
Since $\hat M$ is a Riemannian manifold, so is the sphere bundle 
$\Uu \hat M$. 
Each compact subset of $\Uc\hat M$ goes to a compact subset of 
$\mathds{A}^n$. We can cover a compact fundamental domain of $\Uc \hat M$ 
by finitely many compact convex normal balls $B_i$ in $\Uu \mathds{A}^n$
for $i=1, \dots, f$.  
We define $K_i:=\Uc \hat M \cap B_i$, $i=1, \dots, f$, 
which needs not be connected.
Then we obtain
\begin{equation}\label{eqn:diamB} 
 \dN\hbox{-diam}(\rho'(g)\circ \hat \Pi_{\mathds{A^{n}}}\circ \hat s_{\infty}(K_i)) \leq C
\hbox{ for each } g \in \Gamma_M \hbox{ and } i
\end{equation} 
for $C$ independent of $i$ and $g$. 

Let $L$ be the $d_{\Uu \hat M}$-length of a path $\gamma$ to $\Uc \hat M$. 
We can break $\gamma$ into paths $\gamma_i$, $i=1, \dots, L/\delta'$ of length smaller than the Lebesgue number $\delta'>0$
for the covering $\{B_i\}$. 
Now, $\hat \Pi_{\mathds{A}^n} \circ \hat s_{\infty}\circ \gamma_i$ 
goes into a path in $\mathds{A}^n$ homotopic to 
a path whose length is bounded above by $C$. 
Hence, the image of $\gamma$ is contained in
 a path homotopic to a union of paths 
whose lengths are bounded above by $C$. 
Hence, 
\begin{equation} \label{eqn:upperB-p2} 
 \dN( \hat \Pi_{\mathds{A}^n} \circ\hat s_\infty(x),  \hat \Pi_{\mathds{A}^n} \circ\hat s_\infty(x')) 
\leq \frac{C}{\delta'} d_{\Uu \hat M}(x, x') \hbox{ for } x, x' \in \Uc \hat M. 
\end{equation} 
Hence, $\hat \Pi_{\mathds{A}^n} \circ\hat s_\infty$  is a coarse Lipschitz map. 

We have $\hat \Pi_{\mathds{A}^n} \circ\hat s= \dev\circ \pi_{\Uu \hat M} |\Uc \hat M$ by \eqref{eqn:defines-p2}. 
By the fourth item proved above, we obtain a lower bound on the first term of \eqref{eqn:upperB-p2} 
by 
\[ \dN(\pi_{\Uu \hat M} (x), \pi_{\Uu \hat M} (x')) - 2C'\] 
for the constant $C'$ for the $\dN$-parallel homotopy.
By Lemma \ref{lem:piUUM}  of Part 1, 
$\dev \circ  \pi_{\Uu \hat M}$ is a quasi-isometric embedding. 
Hence, we obtain quasi-isometric embedding $\hat \Pi_{\mathds{A}^n} \circ\hat s_\infty$. 
\end{proof}

%

\subsection{Generalized stable subspaces} 
At each point of $x$ of $\Uc \hat M$, there are vector subspaces
to be denoted by  
$\bV_+(x)$, $\bV_0(x)$, and $\bV_-(x)$ 
respectively corresponding to $\bV_+(p(x))$, $\bV_0(p(x))$, and 
$\bV_-(p(x))$ under the covering 
$\Uc \hat M \times \bR^n \ra \bR^n_{\rho}$. 
Since these are parallel under $\hat \nabla$, 
they are invariant under the geodesic flow $\Phi$ on $\Uc \hat M$
lifting $\phi$.  

Let $\hat s_{\infty}: \Uc \hat M \ra \mathds{A}^n$ be a continuous lift of $s_{\infty}$. 
An affine subspace of $\mathds{A}^n$ parallel to $\bV_0(x, \vec{v})$ passing $\hat s_\infty(x,\vec{v})$ is said to be 
a {\em neutral subspace} of $(x, \vec{v})$.

The first item of Theorem \ref{thm:neutral-p2} implies: 
\begin{corollary} \label{cor:neutralgeo-p2} 
$\hat \Pi_{\mathds{A}^n} \circ \hat s_{\infty}$ restricted to each ray 
$\phi_t(y)$, $t \geq 0$, on $\Uc \hat M$ lies on 
a neutral affine subspace parallel to $\bV_0(\phi_t(y))$ 
independent of $t$.  
\end{corollary} 


From now on, 
\[l_{y} := \{ \phi_{t}(y)| t \geq 0\} \hbox{ for } y \in \Uc \hat M\]
will denote a ray starting from $y$ in $\Uc \hat M$. 
The image $\hat \Pi_{\mathds{A}^n} \circ  \hat s_{\infty}(l_y)$ is in a neutral affine subspace of dimension equal to 
$\dim \bV_0$ by Corollary \ref{cor:neutralgeo-p2}. 
We denote it by $A^0_y$ or $A^0_{l_y}$. 

Since $\dev = \hat \Pi_{\mathds{A}^n} \circ  \hat s$, and 
$\dev \circ \gamma = \rho'(\gamma)\circ \dev$ for $\gamma \in \Gamma_M$, 
we have by an equivariant homotopy
\begin{equation} \label{eqn:sgamma-p2}
\hat \Pi_{\mathds{A}^n} \circ  \hat s_{\infty}\circ \gamma = 
\rho'(\gamma)\circ \hat \Pi_{\mathds{A}^n} \circ  \hat s_{\infty}
\hbox{ for } \gamma \in \Gamma_M.
\end{equation} 
By \eqref{eqn:sgamma-p2}, we obtain
\begin{equation} \label{eqn:sgamma2} 
\rho'(\gamma)(A^0_{l_y})=\rho'(\gamma)(A^0_y) = A^0_{\gamma(y)} = \rho'(\gamma)(A^0_{l_y}) = A^0_{\gamma(l_z)}
\end{equation} 
by Corollary \ref{cor:neutralgeo-p2} and the definition of $A^0_y$. 



Finally, since $s_\infty$ is continuous, the $C^0$-decomposition implies 
that $x \mapsto A^0_x$ is a continuous function. 
Hence, in the Hausdorff metric sense, we obtain
\begin{equation} \label{eqn:ziz-p2} 
A^0_{z_i} \ra A^0_z \hbox{ if } z_i \ra z \in \Uc \hat M. 
\end{equation}

Denote by $V_{+, y}$ the vector subspace parallel to the lift of $\mathds{V}_+$ at $y$. 
Similarly, the $C^0$-decomposition property also implies 
\begin{equation} \label{eqn:ziz2-p2} 
\bV_{e}(z_i) \ra \bV_{e}(z) \hbox{ if } z_i \ra z \in \Uc \hat M \hbox{ for } e= +, -.  
\end{equation}



We will denote for any $q \in \Uu \hat M$ as follows: 
\begin{itemize} 
	\item $A^{e}_q$  
	the affine subspace containing $s_\infty(q)$ and  
	all other points in directions of $\bV_e(q)$ from it for $e=+, -$.
	\item $A^{0 e}_{q}$ 
	the affine subspace containing $A^0_q$ and all other points in 
	directions of $\bV_e(q)$ from points of $A^0_q$ for $e= +, -$. 
\end{itemize} 
We will call $A^{0+}_q$ a {\em generalized unstable affine subspace} 
	and $A^{0-}_q$ the {\em generalized stable affine subspace}.
Again, we have by \eqref{eqn:ziz-p2} and \eqref{eqn:ziz2-p2} 
\begin{equation} \label{eqn:ziz3} 
	A^{0 e}_{z_i} \ra A^{0 e}_{z}   \hbox{ if } z_i \ra z \in \Uc \hat M \hbox{ for } e= +, -.  
\end{equation}   

\begin{figure}[ht!]
	\labellist
	\small\hair 2pt
	\pinlabel $\dev$ [l] at 180 195
	\pinlabel $y_{i+1}$ [t] at 28 133
	\pinlabel $y_i$ [t] at 45 136
	\pinlabel $y_{i-1}$ [t] at 65 140
	\pinlabel $z_{i+1}$ [t] at 27 113
	\pinlabel $z_i$ [t] at 45 108
	\pinlabel $z_{i-1}$ [t] at 65 103
	\pinlabel $\gamma_i$ [t] at 90 130
	\pinlabel $\rho'(\gamma_i)$ [t] at 310 116
	\pinlabel $F$ [t] at 135 95
	\endlabellist
	\includegraphics[width=1.00\textwidth]{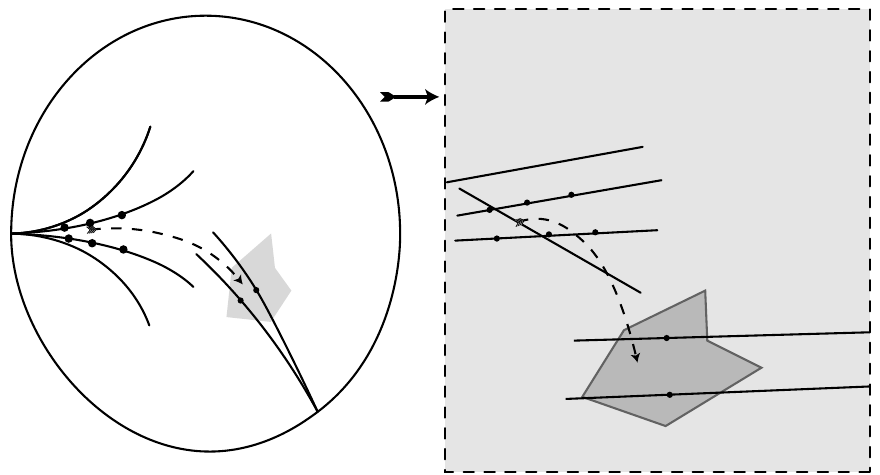}
	\caption[The proof of Proposition \ref{prop:Finalpart-p2}]{The proof of Proposition \ref{prop:Finalpart-p2}}
	\label{fig:pullingback}
\end{figure}



\begin{proposition} \label{prop:Finalpart-p2} 
	Assume that $M$ is an FS submanifold of a complete affine manifold $N$ with word-hyperbolic fundamental group 
	$\Gamma =\Gamma_M$.
Let $p$ be a point of $\partial_\infty \hat M$. 
	Let $y$ be a point of $\mathcal{R}_p$ on a complete isometric geodesic $l_y$ ending at $p$. 
		Suppose that $M$ has a partially hyperbolic linear holonomy homomorphism with respect to a Riemannian metric on $M$ in the bundle sense.  

	Then for every ray $l_z$ in 
	$\mathcal{R}_p$ for $z \in \Uc \hat M$, 
	$\hat \Pi_{\mathds{A}^n} \circ\hat s_\infty(l_z)$ 
	is in single subspace $A^{0-}_{l_y}$. 
	That is, $A^{0-}_{l_z} = A^{0-}_{l_y}$ for every such $l_z$ in $\mathcal{R}_p$, 
	and $\hat \Pi_{\mathds{A}^n} \circ\hat s_\infty(\mathcal{R}_p) \subset A^{0-}_{l_y}$. 
\end{proposition} 
\begin{proof} 
(I) We choose two sequences of points of $\Uu \hat M$ converging towards the ideal point $p$ and find a sequence of deck transformation pulling back to a fundamental domain:  
	Under $\pi_{\Uu \hat M} $, $l_y$ and $l_z$ respectively go to complete geodesics ending at $p$ in the forward direction. 
	Since $\bV^{0e}_{\phi_t(y)}$ are parallel under the flow, 
	$A^{0e}_{\phi_t(y)}$ are independent of $t$ for $e=+, -$.  
Similarly, $A^{0e}_{\phi_t(z)}$ are independent of $t$ for $e=+, -$. 
	
	Choose $y_i \in l_y$ so that $y_i = \phi_{t_i}(y)$, and 
	$z_i \in l_z$ so that $z_i = \phi_{t_i}(z)$ where $t_i \ra \infty$ as $i \ra \infty$. 
	Denote 
	\[ y'_i:= \hat \Pi_{\mathds{A}^n} \circ\hat s_\infty(y_i)
	\hbox{ and } z'_i := \hat \Pi_{\mathds{A}^n} \circ \hat s_\infty(z_i) \hbox{ in } \mathds{A}^n. \]
	We obtain that  
	$d_{\Uu \hat M}(y_i, z_i) < R$
	for a uniform constant $R$ 
	by Lemma 11.75 and Theorem 11.104 of \cite{DK2018} 
	since two bordifications of 
	$\hat M$ agree.
	Since $\hat \Pi \circ \hat s_\infty$ is $\dN$-parallelly homotopic to 
	$\dev \circ \pi_{\Uu \hat M} $ by Theorem \ref{thm:neutral-p2}, 
	we obtain
	\begin{equation}\label{eqn:yizi-p2}  
		 \dN(y'_i, z'_i) < R' 
	\end{equation} 
	for a constant $R'> 0$.
	
	
Since $M$ is compact, 
	$\gamma_i(y_i)$ is in a compact fundamental domain $F$ of $\Uc \hat M$
	for an unbounded sequence $\gamma_i$, $\gamma_i \in \Gamma_M$. 
	$\rho'(\gamma_i)(y'_i)$ is in a compact subset of $\mathds{A}^n$
	for $y_i' = \pi_{\Uu \hat M}  \circ \hat s_{\infty}(y_i)$. 
	Choosing  a subsequence, we may assume without loss of generality 
	\begin{multline} \label{eqn:yinfty-p2} 
		\gamma_i(y_i) \ra y_\infty \hbox{ for a point } y_\infty \in F
		\hbox{ and } \\
		\rho'(\gamma_i)(y'_i) \ra y'_\infty \hbox{ for a point } y'_\infty \in \mathds{A}^n.
	\end{multline}

	\begin{figure}[t!]
		\centering
		\labellist
		\small\hair 2pt
		\pinlabel $y'_{i+1}$ [t] at 26 132
		\pinlabel $y'_i$ [t] at 54 133
		\pinlabel $y'_{i-1}$ [t] at 91 135
		\pinlabel $z'_{i+1}$ [t] at 28 165
		\pinlabel $z'_i$ [t] at 52 165
		\pinlabel $z'_{i-1}$ [t] at 73 166
		\pinlabel $\rho'(\gamma_i)$ [t] at 101 120
		\pinlabel $A^{0-}_{l_y}$ [t] at 115 150
		\pinlabel $A^{0-}_{l_z}$ [t] at 120 170
		\pinlabel $\bV^+(y_i)$ [t] at 58 189
		\pinlabel $\bV^+(y_{i+1})$ [t] at 30 189
		\pinlabel $\bV^+(y_{i-1})$ [t] at 107 189 
		\pinlabel $F$ [t] at 140 90
		\endlabellist
		\includegraphics[width=0.90\textwidth]{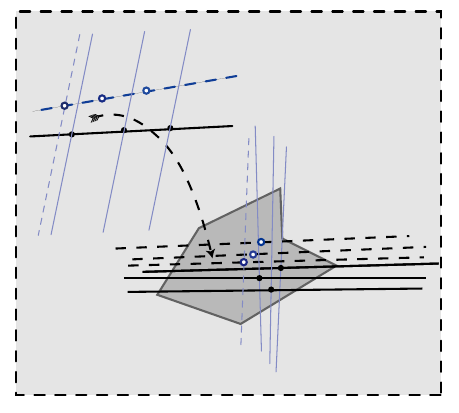}
		\caption{A close-up of the proof of Proposition \ref{prop:Finalpart-p2}.}
		\label{fig:pullingbackII-p2}
		
	\end{figure}

	Since $\gamma_i$ is an isometry of $d_{\hat M} =  \dN| \hat{M}\times \hat{M}$, 
	\eqref{eqn:yizi-p2} shows  
	\begin{equation} \label{eqn:pullback}
		 \dN(\rho'(\gamma_i)(y'_i), 
		\rho'(\gamma_i)(z'_i)) < R'
	\end{equation} 
	as $i \ra \infty$
	for a constant $R'> 0$. 
	Hence, we may assume without loss of generality that 
	\begin{multline} \label{eqn:yinfty2-p2} 
		\gamma_i(z_i) \ra z_\infty \hbox{ for a point } z_\infty \in \Uc \hat M
		\hbox{ and } \\
		\rho'(\gamma_i)(z'_i) \ra z'_\infty \hbox{ for a point } z'_\infty \in \mathds{A}^n.
	\end{multline} 
	%
	%
	%
	%
	
(II) Now we choose the affine subspaces that we need: 
	By Corollary \ref{cor:neutralgeo-p2}, 
    neutral affine subspaces $A^0_{l_y}$ and $A^0_{l_z}$ 
	contain $\hat \Pi_{\mathds{A}^n}(\hat s_{\infty}(y))$ and 
	$\hat \Pi_{\mathds{A}^n}(\hat s_{\infty}(z))$ 
	in $\mathds{A}^n$ respectively. 
	Since the sequence consisting of 
	the $d_{\hat M}$-distances between $\gamma_i(z_i)$ and $\gamma_i(y_i)$ for all $i$ 
	is uniformly bounded above, 
	\eqref{eqn:ziz-p2}, \eqref{eqn:yinfty-p2}, and \eqref{eqn:yinfty2-p2} 
	imply that the sequence of 
	the $d_{\mathcal{AG}_k(\bR^n)}$-distances between 
	\begin{equation} \label{eqn:Hdist}
		A^{0}_{\gamma_i(z_i)} = \rho'(\gamma_i)(A^{0}_{l_z}) \hbox{ and }
		A^{0}_{\gamma_i(y_i)} = \rho'(\gamma_i)(A^{0}_{l_y})
	\end{equation} 
	is bounded above.
	Also, the sequence of the $d_{\mathcal{AG}_k(\bR^n)}$-distances between 
	\begin{equation} \label{eqn:Hdist2-p2} 
		A^{0e}_{\gamma_i(z_i)} = \rho'(\gamma_i)(A^{0e}_{l_z}) \hbox{ and }
		A^{0e}_{\gamma_i(y_i)} = \rho'(\gamma_i)(A^{0e}_{l_y}) \hbox{ for } e = +, -,
	\end{equation} 
	is bounded above. 
	
	Let $\llrrV{\cdot}_E$ denote the norm of 
	the Euclidean metric $d_E$ on $\mathds{A}^n$.

(III)	We claim that $A^{0-}_{l_z}$ is affinely parallel to $A^{0-}_{l_y}$: 
	Suppose not. Then there is a vector $\vec{w}$ at $z$ 
parallel to $A^{0-}_{l_z}$ 
	not parallel to $A^{0-}_{l_y}$. 
Then  $\vec{w}$ has a nonzero component $\vec{w}_+$ in $\bV^+(y_i)$
using the decomposition $\bV_+(y_i) \oplus \bV_0(y_i)\oplus \bV_-(y_i))$  
of $\bV(y_i)$ at $y_i$. 
Working with $y_i$, we obtain that 
the sequence of norms under $\llrrV{\cdot}_\rho$ of  $\rho(\gamma_i)(\vec{w}_+)$ becomes 
	infinite by the condition (iii)(a) of the partial hyperbolicity
	in Definition \ref{defn:phyp}  of Part 1. 
	Since $\gamma_i(y_i)$ is in a compact fundamental domain
	$F$ of $\Uc \hat M$, 
	$\llrrV{\cdot}_\rho$ is uniformly equivalent to the Euclidean norm 
	$\llrrV{\cdot}_E$ associated with $d_E$. 
	Hence, \[\{\llrrV{\rho(\gamma_i)(\vec{w}_+)}_E \} \ra \infty.\] 
	Moreover, by the domination condition (iii)(c) in Definition \ref{defn:phyp}  of Part 1, 
	we obtain that the sequence of 
	directions of $\rho(\gamma_i)(\vec{w})$ converges to 
	that of $\rho(\gamma_i)(\vec{w}_+)$ of $\bV_+(y_\infty)$ under $\llrrV{\cdot}_E$
	up to a choice of a subsequence.  
	
	Also, 
\begin{equation}\label{eqn:disj-p2} 
\bV_+(z_i) \cap (\bV_-(y_i)\oplus \bV_0(y_i)) = \{0\}:
\end{equation}
Suppose not.  We work with $y_i$, and the sequence of 
the images of any nonzero vector $\vec{v}$ in the intersection 
under $\rho(\gamma_i)$ cannot dominate that of 
$\vec{w}_+$  as described in (iii)(c) of Definition \ref{defn:phyp}.
The sequence of $\llrrV{\cdot}_E$-norms of the images of $\vec{w}$ under $\rho(\gamma_i)$ is 
compatible with that of $\vec{w}_+$ again by the same condition. 
The sequence of 
	$\llrrV{\cdot}_E$-norms of the images  under $\rho(\gamma_i)$ of $\vec{v}$ in $\bV_+(z_i)$ cannot dominate those of $\vec{w}$. 
This is absurd since  $\vec{w} \in \bV^{0-}(z_i)$. 

	Hence, every nonzero  vector $\vec{w'}$ in $\bV_+(z_i)$ has a nonzero component parallel to $\bV_+(y_i)$
	under the decomposition $\bV_+(y_i) \oplus \bV_0(y_i) \oplus \bV_-(y_i)$. 

Since $\dim \bV_+(z_i) = \dim \bV_+(y_i)$, we can find a vector $\vec{w}'$ in $\bV_+(z_i)$ with the nonzero component equal to the above vector $\vec{w}_+$ in $\bV_+(y_i)$ by \eqref{eqn:disj-p2}. Hence, the sequence of angles between directions of $\rho(\gamma_i)(\vec{w}')$ and directions of $\rho(\gamma_i)(\vec{w}_+)$ goes to zero 
as $i \to \infty$ by the condition (iii)(c) of Definition \ref{defn:phyp} of Part 1.

The sequence of angles between $\rho(\gamma_i)(A^{0-}_{l_z})$ containing $z'_i$ and $\rho(\gamma_i)(\bV^+(z_i))$ over $z'_i$ converges to zero as $i \to \infty$
by the conclusion of the above paragraph and the fourth paragraph above. 
This contradicts our partial hyperbolic condition (Definition \ref{defn:phyp} of Part 1) since $\{\gamma_i(z_i)\}$ is convergent to a point of $\Uc \hat M$ and the angle between independent $C^0$-subbundles over a compact manifold has a positive lower bound.

	%
	%
	
(IV)	Finally, we show that $A^{0-}_{l_z} = A^{0-}_{l_y}$: Suppose not. 
	Let $\vec{v}$ denote the vector in the direction of $\bV_+(y_i)$ 
	going from parallel affine subspaces $A^{0-}_{l_y}$ to $A^{0-}_{l_z}$. This vector is independent of $y_i$
	since $A^{0-}_{y_i}$ is parallel to $A^{0-}_{l_z} = A^{0-}_{z_i}$.  
	Then for the linear part $A_{\gamma_i}$ of the affine transformation $\gamma_i$, 
	it follows that 
	\[\llrrV{v'_i := A_{\gamma_i}(\vec{v})}_E \ra \infty\]
	by the two paragraphs ago. 
	Since $A^{0-}_{\gamma_i(y_i)} = \rho'(\gamma_i)(A^{0-}_{l_y})$ is 
	fixed under $\gamma_i$, and
	$A^{0-}_{\gamma_i(z_i)} = \rho'(\gamma_i)(A^{0-}_{l_z})$, 
	we have 
	\[K \cap  \rho'(\gamma_i)(A^{0-}_{l_z})= \emp\] for 
	sufficiently large $i$ for every compact subset $K$ of $\hat M$.
	This is a contradiction to the sentence containing \eqref{eqn:Hdist2-p2}.
\end{proof}

\section{Geometric convergences} \label{sec:geoconv-p2}

Now we begin the proof of Theorem \ref{thm:main-p2}. 
\red{For the main idea here, see the outline in the introduction. }


\begin{proposition} \label{prop:quasiiso-p2} Let $p \in \delta_\infty \hat M$. 
Then  $\mathcal{R}_p$ is quasi-isometric to $\hat M$, and 
			there is a $\dN$-cobounded quasi-isometric embedding 
$f:\hat M \ra \mathds{A}^n$ with image in $A^{0-}_y$ 
			for a generalized stable subspace $A^{0-}_y$  with $d_{A^{0-}_y}$
 for any point $y \in \mathcal{R}_p$. 
\end{proposition}
\begin{proof} 
	We can consider $\dev$ an isometry of $d_{\hat M}$ to $ \dN$. 
	We identify $\hat M$ with itself in $\mathds{A}^n$ by $\dev$. 
	So $\dev$ is the inclusion map for this proof. 
We obtain  $\pi_{\Uu \hat M} = \hat \Pi_{\mathds{A}^n} \circ s$, and
	that $\hat \Pi_{\mathds{A}^n}  \circ s: \Uu \hat M \ra \hat M$
	is a quasi-isometry by Lemma \ref{lem:piUUM}  of Part 1. 
	The image $\hat \Pi_{\mathds{A}^n}\circ s(\mathcal{R}_p) = 
	\pi_{\Uu\hat M}(\mathcal{R}_p)$ in $\hat M$ 
	is $C$-dense by Theorem \ref{thm:Ucdense-p2} for $C> 0$. 
	
%

Let $X_p$ denote $\pi_{\Uu \hat M} (\mathcal{R}_p)$. 
The map $\pi_{\Uu \hat M}: {\mathcal{R}}_p \ra X_p$ is a quasi-isometry since 
each fiber for each $x \in \hat M$ is a uniformly bounded set in $\Uu_x \hat M$ 
with metrics $d_{\Uu \hat M}$ and $d_{\mathds{A^n}}$.

By Proposition \ref{prop:Finalpart-p2}, 
$A^{0-}_y = A^{0-}_z$ for every $y, z \in \Uc \hat M$.  
We choose one $A^{0-}_y$. Then under 
$\Pi_{\mathds{A}^n}\circ s_\infty$, every $l_z$ goes into $A^{0-}$ for $z \in 
\mathcal{R}_p$ by Proposition \ref{prop:Finalpart-p2}. 
This fact shows that  
there is a map $\Pi_{\mathds{A}^n}\circ s_\infty : \mathcal{R}_p  \ra A^{0}_y$ is a quasi-isometric embedding
with respect to $d_{\Uu \hat M}$ and $d_{A^{0-}_y}$
by Theorem \ref{thm:neutral-p2}. 
Define a quasi-isometric embedding 
$f: X_p \ra \mathds{A}^{0-}_y$ 
by taking a possibly discontinuous section of $\Pi_{\Uu \hat M}$ and 
post-composing with the above map.

Now, $X_p$ with the restricted metric of $ \dN$ 
is quasi-isometric to $\hat M$ 
by Corollary 8.13 of \cite{DK2018} and Theorem \ref{thm:Ucdense-p2}.
There is the coarse inverse map $\hat M \ra X_p$ to 
the inclusion map $X_p \ra \hat M$. 
Composing $f$ with this map, 
we obtain a quasi-isometric embedding 
$\hat M \ra A^{0-}_y$. 
	\end{proof} 

\begin{corollary} 
	$\pi_1(N)$ quasi-isometrically embeds into a generalized stable affine subspace.
	\end{corollary} 
\begin{proof} 
	Since $\hat M$ is quasi-isometric with an orbit of $\pi_1(N)$, this follows. 
	\end{proof}

We use the theorem of Block-Wienberger combining their footnote comment. 

\begin{proposition}[Connect-the-dots in Block-Weinberger \cite{BW93}] \label{prop:BW-p2}
	Suppose that $f:Z \ra A$ is a coarse Lipschitz map from a finite-dimensional polyhedron $Z$ to 
	a metric subspace $A$ uniformly contractible in a metric space $B$, $A \subset B$,
	Let $Z'\subset Z$ be a subcomplex. Suppose that $f|Z'$ is continuous. 
Suppose that there is a cell-decomposition of $Z$ so that the image of each cell under $f$ in $Z$ has a bounded diameter $\leq C'$.	
Then $f$ is of a $C(C')$-bounded distance from a 
	continuous Lipschitz map $f':Z \ra B$
	where $f'|Z' = f|Z'$ and $C(C')$ is a constant depdning only on $C'$. 
\end{proposition} 
\begin{proof} 
	We simply extend $f$ over each cell using the uniform contractibility
	as indicated in \cite{BW93} starting from $0$-cells and then to $1$-cells and so on. 
Here, we do need the boundedness of the diameters of each cell to obtain the bounded distance.
\end{proof} 

Let $Z$ be a metric space. 
Let $H^j_{c}(X), j\in \bZ$ denote the direct limit 
\[\varinjlim H^j(X, X-K) \]
where $K$ is a compact subset of $X$ partially ordered by inclusion maps. (See Hatcher \cite{Hatcher}.)

Given two chain complexes $(C, d)$ and $(C', d')$, 
we define the {\em function complex} $\mathpzc{Hom}(C, C')$ by defining 
$\mathpzc{Hom}(C, C')_e$ to be the set of graded module homomorphisms of degree $e$. 
(See page 5 of \cite{Brown}.)

\begin{proof}[Proof of Theorem \ref{thm:main-p2}.] 
Suppose that $\rho|\Gamma_M$ is partially hyperbolic representation in the bundle sense
with index $k$ for $k < n/2$. By Proposition \ref{prop:Uc} of Part 1, 
$\rho$ is a $k$-Anosov representation
in the bundle sense according to 
the definition in Section 4.2 of \cite{BPS}. Proposition 4.5 of \cite{BPS} implies that 
$\rho$ is $k$-dominated. By Theorem 3.2 of \cite{BPS} (following from 
Theorem 1.4 of \cite{KLP18}), $\Gamma_M$ is 
word hyperbolic.

There exists an exhaustion of $\mathds{A}^n/\Gamma$ by compact FS submanifolds 
$M_i$ where $M=M_1 \subset M_2 \subset \cdots$. 
(See Scott-Tucker \cite{ST89} for constructions. Here, FS property easy to obtain.)
Also, we may choose a Riemannian metric so that each $M_i$ has convex boundary. 
Let $\hat M_i$ denote the cover of $M_i$ in $\mathds{A}^n$. 


(I) The first step is to homotopy $\dN$-parallelly the inclusion of $\hat M_i \ra \mathds{A}^n$ 
to a $\dN$-cobounded quasi-isometry into an affine subspace of dimension $n-k$ using Theorem \ref{thm:k-conn-p2}:  

Let us fix $i$ to start. 
Since $\Gamma_M$ is word-hyperbolic, 
we can apply all the results in the previous sections.
Proposition \ref{prop:quasiiso-p2}  gives us a $\dN$-cobounded quasi-isometric embedding 
$f: \hat M_i \ra \mathds{A}^n$ with the image in 
$N_C(\hat M_i) \cap L$ for an affine subspace $L$ of dimension $n-k$. 
Here, $N_C(\hat M_i)$ is a $C$-neighborhood of $\hat M_i$ in $\mathds{A}^n$
for some $C$ where $C$ is the constant obtained by Theorem \ref{thm:neutral-p2}
since we are modifying the map by neutralization. 
Since $M_i$ is compact, and $f$ is $\dN$-cobounded, 
each image of a cell under $f$ in $L$ has a uniformly bounded $\dN$-diameter.
Hence, each has a uniformly bounded $d_L$-diameter by the last part of 
Theorem \ref{thm:k-conn-p2}.

%
Since $N_C(\hat M_i)\cap L$ is uniformly contractible in $L$ with respect to $d_L$ by Theorem \ref{thm:k-conn-p2}, 
we modify $f$ to be a continuous quasi-isometric embedding to 
$N_{C'}(\hat M_i)\cap L$ by Proposition \ref{prop:BW-p2}
for another constant $C'$ depending only on $C$. 
(Of course here, we will not try to extend 
any already continuous parts defined on $\mathcal{R}_p$ and start from $0$-cells. )
We let $C$ to denote $C'$ from now on. 

Now, $f$ as a map to $\mathds{A}^n$ is $\dN$-cobounded  
since we modified the original $\dN$-cobounded map in a bounded manner in $L$ with respect to $d_L$ using Proposition \ref{prop:BW-p2} and 
$(L, d_L) \ra (\mathds{A}^n, \dN)$ is distance-nonincreasing. 
Using the inclusion map $\iota:   N_C(\hat M_i) \cap L \ra N_C(\hat M_i) $, we have
\begin{equation}
	\hat M_i \stackrel{f}{\ra}  N_C(\hat M_i) \cap L \stackrel{\iota}{\hookrightarrow} N_C(\hat M_i) \hookrightarrow  \hat M_{j(i)}
\end{equation}
for a sufficiently large $j(i)$. 
Denote the composition of the right two maps by $\iota$ also. 
Since $\iota \circ f$ is $\dN$-cobounded  in terms of $ \dN$,
%
there is a $\dN$-parallel homotopy $H$ between 
\[\iota \circ f: \hat M_i \ra \hat M_{j(i)} \hbox{ and } \iota_{ij(i)}: \hat M_i \ra \hat M_{j(i)}\]
by Lemma \ref{lem:parallel-p2} up to choosing $j(i)$ bigger  again
to accommodate the $\dN$-parallel homotopy. 
This is equivariant homotopy, and for each $t\in [0, 1]$. 
We may assume that the image of $H$ is in $\hat M_{j(i)}$ by taking sufficiently large $j(i)$. 

We note that $H, \iota, f, \iota_{ij(i)}$ are all continuous. 

(II) The next step is to apply the homotopy to the cohomology theory to compute the cohomological dimensions:

We recall three facts from \cite{Brown} 
\begin{itemize} 
\item 
$\mathpzc{Hom}(C, C')$ for chain complexes $(C, d)$ and $(C', d')$ is the set of graded module 
homomorphisms of degree $n$ from $C$ to $C'$. 
That is   \[\mathpzc{Hom}(C, C')_n:= \prod_{q\in \bZ}  \Hom(C_q, C'_{q+n})\] with  
the boundary operator $D_n$ is given by $D_n(f) := d' f - (-1)^n fd$.
(See Page 5 of \cite{Brown}.) 

\item 
$\mathpzc{Hom}_c(C_\ast(\hat M_i), \bZ) \subset \mathpzc{Hom}(C_\ast(\hat M_i), \bZ)$ 
consists of every chain map $f: C_\ast(\hat M_i) \ra \bZ$ such that for 
each $m \in M$, $f(\gamma m) =0$ for all but finitely many $\gamma \in \Gamma$. 
We define $H^j_c*(\hat M_i)$ to be the $j$-th cohomology of 
$\mathpzc{Hom}_c(C_\ast(\hat M_i), \bZ)$. 
(See Section 7 of Chapter 8 of \cite{Brown}.)  
\item Since $K(\Gamma, 1)$ is realized as a finite complex, $\Gamma$ is of type FL by Proposition 6.3 of Chapter 8 of  \cite{Brown}.  
By Proposition 6.7 of Chapter 8 of  \cite{Brown}, we have 
\begin{equation}\label{eqn:cd} 
\mathrm{cd}\Gamma = \max\{j | H^j(\Gamma; \bZ\Gamma)\ne 0\}.
\end{equation} 
\end{itemize} 
 
(II)(a) We have maps
\[H^j_c(\hat M_{j(i)}) \stackrel{\iota^*}{\ra} H^j_c(L\cap N_C(\hat M_i)) \stackrel{f^*}{\ra} H^j_c(\hat M_i)
\hbox{ for each } j \in \bZ.\] 
The composition equals $\iota_{ij(i)}^\ast$ by the $\dN$-parallel homotopy $H$. 
Since $L \cap N_c(\hat M_i))$ is a finite-dimensional smooth manifold, 
$\iota_{ij(i)}^*$ is zero for dimensions $> \dim L$. 
Now, we choose a subsequence of $M_i$ relabeled so that $M_{i+1} = M_{j(i)}$ for each $i$, $i=1, 2, \dots,$
where $j(i)$ is chosen as above. 
Therefore, we obtain
\begin{equation}\label{eqn:dimL-p2} 
\iota_{ij}^{\ast k}=0 \hbox{ for } k> \dim L, i < j. 
\end{equation}


(II)(b) Let $\tilde K(\Gamma, 1)$ denote the universal cover of $K(\Gamma, 1)$. According to the top of page 209 of \cite{Brown}, $H^*(\Gamma, \mathbb{Z}\Gamma)$ is the cohomology of $\mathpzc{Hom}_{\Gamma}(C_*(\tilde K(\Gamma, 1)), \mathbb{Z}\Gamma)$. $\mathbb{A}^n$ is a contractible free $\Gamma$-complex of $X$. Since $\mathbb{A}^n/\Gamma$ is homotopy equivalent to $K(\Gamma, 1)$, there exist maps
\[ f_1: \mathbb{A}^n \to \tilde K(\Gamma, 1) \quad \text{and} \quad f_2: \tilde K(\Gamma, 1) \to \mathbb{A}^n \]
such that $f_1 \circ f_2$ and $f_2 \circ f_1$ are homotopic to the identity maps equivariantly with respect to the $\Gamma$-actions. Hence, $C_*(\mathbb{A}^n)$ is chain homotopy equivalent to a finite free resolution $C_*(\tilde K(\Gamma, 1))$ of $\mathbb{Z}$ in a $\mathbb{Z}\Gamma$-equivariant manner with respect to $\mathbb{Z}\Gamma$-actions. Therefore, $H^*(\Gamma, \mathbb{Z}\Gamma)$ equals the domain of the isomorphism
\[ f_2^*: H^*(\mathpzc{Hom}_{\Gamma}(C_*(\mathbb{A}^n), \mathbb{Z}\Gamma)) \to H^*(\mathpzc{Hom}_{\Gamma}(C_*(\tilde K(\Gamma, 1)), \mathbb{Z}\Gamma)). \]

%

(II)(c) Since $\hat M_i$ exhausts $\mathds{A}^n$, 
$C_\ast(\mathds{A}^n)$ equals $\varinjlim C_\ast(\hat M_i)$ as $\bZ\Gamma$-modules. 
We have 
\begin{equation} \label{eqn:dlimit} 
\mathpzc{Hom}_{\Gamma}(C_*(\mathds{A}^n), \bZ\Gamma) = \varprojlim  \mathpzc{Hom}_{\Gamma}(C_\ast(\hat M_i), \bZ\Gamma).
\end{equation} 
Let $\tilde \iota_i: \hat M_i \ra \mathds{A}^n$ be the lift of the inclusion map $\iota_i: M_i \ra N$. The lift is unique since $\pi_1(M_i) \ra \pi_1(N)$ is surjective. 
Then we have
\begin{equation} \label{eqn:varp2-p2}
	\Lambda_i =\tilde \iota_i^\ast: H^l(\mathpzc{Hom}_{\Gamma}(C_*(\mathds{A}^n), \bZ\Gamma)) \ra 
	H^l( \mathpzc{Hom}_{\Gamma}(C_\ast(\hat M_i), \bZ\Gamma)) \hbox{ for all } l.
\end{equation} 
By Theorem 3.5.8 of \cite{Weibel}, there is a surjective homomorphism 
\begin{equation} \label{eqn:varp-p2}
\Lambda: H^l(\mathpzc{Hom}_{\Gamma}(C_*(\mathds{A}^n), \bZ\Gamma)) \ra 
\varprojlim H^l( \mathpzc{Hom}_{\Gamma}(C_\ast(\hat M_i), \bZ\Gamma)) \hbox{ for all } l.
\end{equation} 
where $\Lambda$ is the inverse limit of $\Lambda_i$.  
We may assume that the image of $f_2$ is in $\hat M_i$ for all $i$. 
Let $f_2^i: \tilde K(\Gamma, 1) \ra \hat M_i$  denote the restriction of the range space. 
Then  we have $\tilde \iota_i \circ f_2^i = f_2$, and 
$f_2^* = f_2^{i*} \circ \Lambda_i$ is an isomorphism. 
This means that each $\Lambda_i$ is injective. Hence, we deduced that $\Lambda$ is an isomorphism. 

(II)(d) By Lemma 7.4 of Chapter 8 of \cite{Brown}, there is a natural isomorphism
\red{
\[\mathpzc{Hom}_{\Gamma}(C_\ast(\hat M_i), \bZ\Gamma) \cong \mathpzc{Hom}_c(C_\ast(\hat M_i), \bZ).\] 
}
Since the cohomology of 
$\mathpzc{Hom}_c(C_\ast(\hat M_i), \bZ)$ is 
$H^*_c(\hat M_i)$, the right side of \eqref{eqn:varp-p2} is zero for $l > \dim L$
 by \eqref{eqn:dimL-p2}.  
By the conclusion of (II)(b), we obtain that 
\[  H^l(\Gamma, \bZ\Gamma) = 0 \hbox{ for } l > \dim L.\]
By \eqref{eqn:cd}, we obtain
$\mathrm{cd}(\Gamma) \leq \dim L$. 
%
\end{proof}

\begin{proof}[Proof of Corollary \ref{cor:cpt-p2}]
By Theorem \ref{thm:main2} of Part 1, $\rho$ is partially hyperbolic and $k < n/2$. There is a constant $C$ of the $\dN$-parallel homotopy obtained in Theorem \ref{thm:neutral-p2}. For each $z \in \partial_\infty \pi_1(N)$, $\mathcal{R}_z$ embeds quasi-isometrically into a generalized stable affine subspace $A_z$ by lifting points to $\Uc \tilde{M}$ and then applying $\tilde{s}_\infty$, moving only a distance bounded above by $C$ by Proposition \ref{prop:quasiiso-p2}. Let $K$ be a compact fundamental domain of $\Uc \tilde{M}$. Since $\mathcal{R}_z$ is $C'$-dense for some $C' > 0$, a point of $\mathcal{R}_z$ is in a $C'$-neighborhood $K'$ of $K$. The affine subspace $A_z$ is determined by $\tilde{s}_\infty(K' \cap \mathcal{R}_z)$ for each $z \in \partial_\infty \pi_1(N)$. Since $\tilde{s}_\infty(K')$ is compact, and each $A_z$ for $z \in \partial_\infty \pi_1(N)$ contains a point of $\tilde{s}_\infty(K')$, the set of $A_z$ for $z \in \partial_\infty \pi_1(M)$ is compact. Since all $A_z$ for $z \in \partial_\infty \pi_1(M)$ are considered, the invariance under the affine group follows.

\end{proof} 

\begin{remark}
	Provided $M$ is closed, 
we may have assumed in the above proof that  
the linear part homomorphism $\rho:\Gamma \ra \GL(n, \bR)$ 
is injective by Corollary 1.1 of Bucher-Connel-Lafont \cite{BCL}
since $\Gamma$ is hyperbolic and hence the simplicial volume is nonzero
by Gromov \cite{Gromov82}. 
\end{remark}

\bibliographystyle{plain}
\bibliography{AA-PI}   

\end{document}